\documentclass[11pt]{article}

\usepackage[margin=1in]{geometry}
\usepackage[T1]{fontenc}
\usepackage[utf8]{inputenc}
\usepackage{lmodern}
\usepackage{microtype}
\usepackage{amsmath,amssymb,amsthm}
\usepackage{tikz}
\usepackage{booktabs}
\usepackage{enumitem}
\usepackage{xcolor}
\usepackage[hidelinks]{hyperref}
\usepackage{aliascnt}
\usepackage[nameinlink,noabbrev]{cleveref}

\setlist[itemize]{itemsep=0.2em, topsep=0.3em}
\setlist[enumerate]{itemsep=0.2em, topsep=0.3em}
\allowdisplaybreaks

\DeclareMathOperator{\Spec}{Spec}

\DeclareMathOperator{\Hom}{Hom}

\DeclareMathOperator{\Hilb}{Hilb}
\DeclareMathOperator{\Supp}{Supp}
\DeclareMathOperator{\Soc}{Soc}
\DeclareMathOperator{\Bl}{Bl}

\DeclareMathOperator{\length}{length}
\DeclareMathOperator{\res}{res}

\DeclareMathOperator{\Span}{Span}

\newcommand{\CC}{\mathbb C}

\newcommand{\NN}{\mathbb N}
\newcommand{\Aff}{\mathbb A}
\newcommand{\PP}{\mathbb P}

\newcommand{\eps}{\varepsilon}

\newcommand{\Hilbn}[1]{(\Aff^2)^{[#1]}}
\newcommand{\Hilbnest}[2]{(\Aff^2)^{[#1,#2]}}

\theoremstyle{plain}
\newtheorem{theorem}{Theorem}[section]
\newaliascnt{proposition}{theorem}
\newtheorem{proposition}[proposition]{Proposition}
\aliascntresetthe{proposition}
\newaliascnt{lemma}{theorem}
\newtheorem{lemma}[lemma]{Lemma}
\aliascntresetthe{lemma}

\newaliascnt{corollary}{theorem}
\newtheorem{corollary}[corollary]{Corollary}
\aliascntresetthe{corollary}

\theoremstyle{definition}
\newaliascnt{definition}{theorem}
\newtheorem{definition}[definition]{Definition}
\aliascntresetthe{definition}
\newaliascnt{example}{theorem}
\newtheorem{example}[example]{Example}
\aliascntresetthe{example}
\newaliascnt{convention}{theorem}
\newtheorem{convention}[convention]{Convention}
\aliascntresetthe{convention}

\theoremstyle{remark}
\newaliascnt{remark}{theorem}
\newtheorem{remark}[remark]{Remark}
\aliascntresetthe{remark}

\newcommand{\yboxsize}{0.42}

\newcommand{\drawyoungmarked}[3]{%
\begin{tikzpicture}[baseline={(0,-0.15)}]
  \foreach \x/\y in {#1} {
    \draw[thick] (\x*\yboxsize,\y*\yboxsize) rectangle ++(\yboxsize,\yboxsize);
  }
  \fill[gray!35] (#2*\yboxsize,#3*\yboxsize) rectangle ++(\yboxsize,\yboxsize);
  \draw[thick] (#2*\yboxsize,#3*\yboxsize) rectangle ++(\yboxsize,\yboxsize);
\end{tikzpicture}%
}

\newcommand{\drawyoungaxes}[1]{%
\begin{tikzpicture}[baseline={(0,-0.15)}]
  \draw[->] (-0.2,0) -- (2.7,0) node[right] {$i$};
  \draw[->] (0,-0.2) -- (0,2.7) node[above] {$j$};
  \foreach \x/\y in {#1} {
    \draw[thick] (\x*\yboxsize,\y*\yboxsize) rectangle ++(\yboxsize,\yboxsize);
  }
\end{tikzpicture}%
}

\crefname{theorem}{Theorem}{Theorems}
\Crefname{theorem}{Theorem}{Theorems}
\crefname{proposition}{Proposition}{Propositions}
\Crefname{proposition}{Proposition}{Propositions}
\crefname{lemma}{Lemma}{Lemmas}
\Crefname{lemma}{Lemma}{Lemmas}
\crefname{corollary}{Corollary}{Corollaries}
\Crefname{corollary}{Corollary}{Corollaries}
\crefname{definition}{Definition}{Definitions}
\Crefname{definition}{Definition}{Definitions}
\crefname{remark}{Remark}{Remarks}
\Crefname{remark}{Remark}{Remarks}
\crefname{section}{Section}{Sections}
\Crefname{section}{Section}{Sections}

\title{Deformation Theory and Torus-Fixed Geometry of\\ the Nested Hilbert Scheme of Points}

\author{Chenyang Zhao\thanks{Department of Mathematics, Imperial College London, United Kingdom.\ \textit{Email:} \texttt{cz2922@ic.ac.uk}}}

\date{}

\begin{document}
\maketitle

\begin{quote}
\centering
``All arts among the human race are from Prometheus.''\\[0.3em]
\emph{Aeschylus, Prometheus Bound, l. 506}
\end{quote}

\begin{abstract}
In this paper, we give a self-contained, expository account of the nested Hilbert scheme
\[
  \Hilbnest{n}{n+1}=\Hilb^{n,n+1}(\Aff^2)
\]
from a combination of deformation theory, torus actions, and Young diagram combinatorics. We first recall the scheme theory and functor basics needed to define Hilbert schemes. We then use a classic result on first order deformations to identify
\[
  T_I\Hilbn{n}\cong \Hom_{\CC[x,y]}(I,\CC[x,y]/I).
\]
For a nested pair $I\subset J$, with $\dim_\CC \CC[x,y]/I=n+1$ and $\dim_\CC \CC[x,y]/J=n$, the tangent space becomes a compatibility kernel
\[
T_{(I,J)}\Hilbnest{n}{n+1}
\cong
\ker\!\left(
\Hom(I,R/I)\oplus \Hom(J,R/J)
\longrightarrow
\Hom(I,R/J)
\right).
\]
The torus-fixed points are indexed by a partition $\lambda\vdash n+1$ together with a removable corner $c$ of its Young diagram. This corner is not only combinatorial, but also the monomial form of a one dimensional socle direction in $R/I_\lambda$.  The blow-up map to $\Hilbn{n}\times\Aff^2$ has fibres given by projective spaces of one dimensional quotients of $J/\mathfrak m_pJ$, whose torus-fixed points are addable boxes of the smaller diagram. These two local fibres explain how the universal family, the blow-up geometry, and Young diagram combinatorics come together in the study of the local geometry of the nested Hilbert scheme of points. Finally, we derive the tangent weight formula at a fixed point $(I_\lambda,I_{\lambda\setminus c})$ in the torus convention used here. Using the standard arrow basis, we show in the proof how the arm-leg weights are modified by the compatibility kernel through a shortening rule determined by $c$. A Macaulay2 verification computes the compatibility kernel from monomial syzygies and checks the weight formula for all partitions of size at most $16$.
\end{abstract}

\section*{Acknowledgement}
This paper is based on my master's thesis, completed at Imperial College London under the supervision of Dr Matt Booth, whom I thank for his generous guidance and continuous support. I acknowledge the use of ChatGPT for language editing, grammar fixing, LaTeX formatting, and generating the Young diagrams and arrows. All the mathematical content is my own work.

\tableofcontents

\section{Introduction}
\label{chap:introduction}

Hilbert schemes of points are moduli spaces in algebraic geometry. If $X$ is a smooth surface, the Hilbert scheme $X^{[n]}$ parametrizes zero dimensional closed subschemes of $X$ of length $n$. It contains the open locus of $n$ distinct reduced points, but also nonreduced schemes supported at fewer points. Fogarty's theorem says that $X^{[n]}$ is smooth of dimension $2n$ when $X$ is a smooth surface \cite{Fogarty1968}. This smoothness is one reason why Hilbert schemes of points on surfaces are important examples in algebraic geometry and also interact with representation theory and symmetric combinatorics \cite{Nakajima1999,Haiman2001}.

This paper studies the first nontrivial nested Hilbert scheme of points on the affine plane. We write
\[
  R=\CC[x,y],
  \qquad
  S=\Aff^2_\CC=\Spec R.
\]
The nested Hilbert scheme
\[
  S^{[n,n+1]}=\Hilb^{n,n+1}(S)
\]
parametrizes pairs of zero dimensional closed subschemes
\[
  Y\subset Z\subset S,
  \qquad
  \length(Y)=n,
  \qquad
  \length(Z)=n+1.
\]
Equivalently, it parametrizes pairs of ideals
\[
  I\subset J\subset R,
  \qquad
  \dim_\CC R/I=n+1,
  \qquad
  \dim_\CC R/J=n.
\]
The inclusion of ideals is opposite to the inclusion of subschemes, hence $I$ defines the larger subscheme $Z$, while $J$ defines the smaller subscheme $Y$.

Nested Hilbert schemes on surfaces appear frequently in enumerative geometry. For example, Gholampour--Sheshmani--Yau construct virtual fundamental classes for nested Hilbert schemes of points and curves on nonsingular projective surfaces \cite{GholampourSheshmaniYau2020}.

The local geometry of $S^{[n,n+1]}$ can be studied by combining deformation theory with the combinatorics of torus-fixed monomial ideals. A tangent vector to the ordinary Hilbert scheme is a first-order deformation of an ideal. For an ideal $I\subset R$ of colength $n$, the standard deformation theory identification is
\[
  T_I S^{[n]}\cong \Hom_R(I,R/I),
\]
see \cite[Proposition~3.2.1]{Sernesi2006}. For a nested pair $I\subset J$, the two first-order deformations must be compatible with the inclusion. This gives the usual kernel formula for the incidence condition, compare Cheah's tangent space formulation for nested Hilbert schemes \cite[Section~0.4]{Cheah1998} and Can's kernel formula \cite[Section~5]{Can2007}:
\[
T_{(I,J)}S^{[n,n+1]}
\cong
\ker\left(
\Hom_R(I,R/I)\oplus \Hom_R(J,R/J)
\xrightarrow{\delta}
\Hom_R(I,R/J)
\right),
\]
where
\[
  \delta(\alpha,\beta)=\pi\circ \alpha-\beta\circ\iota.
\]
Here $\iota:I\hookrightarrow J$ is the inclusion map and $\pi:R/I\to R/J$ is the quotient map. This compatibility kernel, proved in \Cref{thm:nested-kernel}, is the main link between deformation theory and local geometry of the nested Hilbert scheme in this paper.

The two dimensional torus
\[
  T=(\CC^\times)^2
\]
acts on $S$ by scaling the two affine coordinates, and hence acts on both $S^{[n]}$ and $S^{[n,n+1]}$. The fixed points of $S^{[n]}$ are monomial ideals, indexed by partitions or Young diagrams. At a monomial ideal $I_\lambda$, the ordinary tangent weights are the standard arm-leg weights, and Nakajima described the formula in an arm-leg form \cite[Proposition~5.8]{Nakajima1999}, and Haiman uses the same weights in equivariant localization \cite[Section~3]{Haiman1998}. For the nested scheme, a torus-fixed point is indexed by a partition $\lambda\vdash n+1$ together with a removable corner $c$ of its Young diagram.

The corner is not only a combinatorial structure, it also has an algebraic meaning. If $I\subset J$ and $J/I$ has length one supported at a point $p$, then $J/I$ is a one dimensional submodule of $R/I$ killed by the maximal ideal $\mathfrak m_p$. Equivalently, it is a line in the local socle
\[
  \Soc_p(R/I)=\{\bar f\in R/I: \mathfrak m_p\bar f=0\}.
\]
For a monomial ideal $I_\lambda$ supported at the origin, the monomial socle directions are exactly the removable corners of $D(\lambda)$. Dually, the blow-up map to $S^{[n]}\times S$ has fibres described by projective spaces of one dimensional quotients of $J/\mathfrak m_pJ$, and for monomial ideals, the torus-fixed quotient directions are addable boxes of the smaller diagram. These two fibre descriptions explain why removable corners, addable boxes, and first-order nested deformations are different ways of describing the same incidence condition.

This paper has three expository goals. First, it compares the birational map from $S^{[n,n+1]}$ to the universal family over $S^{[n+1]}$ with the blow-up map to $S^{[n]}\times S$. Second, it works through the corresponding local fibre descriptions in terms of socle lines and one dimensional quotients of $J/\mathfrak m_pJ$. Third, it gives a deformation-theoretic account of the shortening rule for nested tangent weights in the torus convention used here, via the compatibility kernel. None of these results is new: the tangent space computation is due to Cheah \cite[Section~2.6]{Cheah1998}, Chaput--Evain studied the same geometry using staircases and cleft pairs \cite[Sections~2.1--2.2]{ChaputEvain2015}, and Koncki--Zielenkiewicz explained the shortening rule in arrow notation \cite[Section~5.3]{KonckiZielenkiewicz2025}.

We also include a Macaulay2 verification \cite{Macaulay2}. We compute the compatibility kernel from monomial generators and syzygies, then compare the resulting weight multiset with the Young diagram shortening rule for all partitions of size at most $16$.

\Cref{chap:background} introduces scheme theory and defines Hilbert schemes of points and their first-order deformations with a brief introduction of deformation theory.  \Cref{chap:young-monomial} introduces partitions, Young diagrams, monomial ideals, corners, and addable boxes.  \Cref{chap:hilb-plane} then introduces the torus-fixed geometry of the Hilbert scheme $S^{[n]}$ on the affine plane, including the torus action and the arm-leg tangent weights.  \Cref{chap:nested-geometry} studies the nested scheme $S^{[n,n+1]}$, its relation to the universal family and the blow-up of $S^{[n]}\times S$ and the local fibre.  \Cref{chap:nested-tangent} explains the compatibility kernel formula for tangent spaces and the shortening rule for torus weights. Finally, \Cref{chap:examples-computation} gives explicit examples and explains the Macaulay2 verification.

\section{Schemes, Hilbert schemes, and first-order deformations}
\label{chap:background}

This section introduces basic scheme theory, explains how Hilbert schemes represent flat families, why Hilbert schemes of points are the relevant moduli spaces here, and why their tangent spaces are equivalent to homomorphisms of ideals. References for scheme theory are Hartshorne \cite[Chapter~II]{Hartshorne1977} and Grothendieck's construction of Hilbert schemes \cite{GrothendieckHilbert1961}, and for Hilbert schemes of points on surfaces we follow the notation of Nakajima \cite[Chapter~1]{Nakajima1999}.

In this paper all schemes are over $\CC$. The affine plane is denoted
\[
  S=\Aff^2_\CC=\Spec R,
  \qquad R=\CC[x,y].
\]

\subsection{Affine schemes and closed subschemes}

For a commutative ring $A$, the affine scheme $\Spec A$ is the set of prime ideals of $A$ equipped with the Zariski topology and the structure sheaf $\mathcal O_{\Spec A}$, see \cite[Chapter~II, Section~2]{Hartshorne1977}. If $A$ is a finitely generated $\CC$-algebra, then $\Spec A$ is the analogue of an affine algebraic variety. Hilbert schemes remember nonreduced subschemes. For example, the ideals $(x,y)$ and $(x,y)^2$ have the same underlying support in $\Aff^2$, namely the origin, but they define different closed subschemes.

A closed subscheme of $\Spec A$ is defined by an ideal $I\subset A$, see \cite[Chapter~II, Section~3]{Hartshorne1977}. It is
\[
  \Spec(A/I)\subset \Spec A.
\]
The underlying set is the closed set $V(I)$, but the quotient ring $A/I$ also contains nilpotent and infinitesimal structure. An element $\bar f\in A/I$ is nilpotent if $\bar f^N=0$ for some $N>0$, and such elements are the algebraic form of infinitesimal structure.

For example,
\[
  \CC[x,y]/(x,y)\cong \CC,
\]
whereas
\[
  \CC[x,y]/(x,y)^2
  =
  \Span_\CC\{1,x,y\}.
\]
Thus $(x,y)^2$ defines a fat point of length three supported at the origin, not a reduced point. In particular, a zero dimensional closed subscheme of $\Aff^2_\CC$ is defined by an ideal $I\subset \CC[x,y]$ such that $R/I$ is a finite dimensional $\CC$-vector space.

\begin{definition}
Let $Z$ be a zero dimensional scheme over $\CC$. Its length is
\[
  \length(Z)=\dim_\CC H^0(Z,\mathcal O_Z).
\]
If $Z\subset \Aff^2_\CC$ is defined by an ideal $I\subset R$, then
\[
  \length(Z)=\dim_\CC R/I.
\]
\end{definition}

Thus a point of the Hilbert scheme of $n$ points on $\Aff^2_\CC$ is represented by an ideal $I\subset R$ with $\dim_\CC R/I=n$.

\subsection{The functor of points}
\label{sec:functor-points}

A scheme $X$ can be studied through its functor of points
\[
  h_X:(\CC\text{-schemes})^{\mathrm{op}}\longrightarrow \mathrm{Sets},
  \qquad
  h_X(B)=\Hom_{\CC}(B,X).
\]
An element of $h_X(B)$ is called a $B$-valued point of $X$. When $B=\Spec \CC$, this recovers the ordinary closed points over $\CC$. When $B$ is nonreduced, a $B$-valued point can carry infinitesimal information. This point of view is standard in moduli theory, see \cite[Chapter~II, Section~2]{Hartshorne1977}.

For example, let
\[
  D=\Spec \CC[\eps]/(\eps^2).
\]
A morphism $D\to X$ whose closed point maps to $x\in X$ is a tangent vector to $X$ at $x$. This is the reason that tangent spaces and first-order deformations are naturally expressed using the dual numbers.

Moduli spaces are most naturally defined by functors. A scheme $M$ represents a moduli functor $F$ if there is a natural bijection of sets
\[
  F(B)\cong \Hom(B,M)
\]
for every test scheme $B$. In this case, families over $B$ are the same thing as morphisms $B\to M$.

\subsection{Hilbert polynomials and flat families}

Let $X$ be a projective scheme over $\CC$, equipped with a closed embedding $X\subset \PP^N$. The embedding gives twisting sheaves $\mathcal O_X(m)$, and for a coherent sheaf $\mathcal F$ on $X$ we write $\mathcal F(m)=\mathcal F\otimes\mathcal O_X(m)$. Its Euler characteristic is
\[
  \chi(X,\mathcal F(m))=
  \sum_i(-1)^i\dim_\CC H^i(X,\mathcal F(m)).
\]
For all sufficiently large integers $m$, the function
\[
  m\longmapsto \chi(X,\mathcal F(m))
\]
agrees with a polynomial, called the Hilbert polynomial of $\mathcal F$, see \cite[Chapter~III, Section~5 and Section~9]{Hartshorne1977}. If $Z\subset X$ is a closed subscheme, the Hilbert polynomial of $Z$ means the Hilbert polynomial of $\mathcal O_Z$.

A family of subschemes over a base $B$ is a closed subscheme
\[
  \mathcal Z\subset X\times B.
\]
We require that $\mathcal Z$ be flat over $B$. Flatness prevents the Hilbert polynomial from jumping in families, and for a flat projective family the Hilbert polynomial of the fibres is locally constant on the base, see \cite[Chapter~III, Section~9]{Hartshorne1977}.

For example, let $X=\Aff^1_x$ and $B=\Aff^1_t$, and consider the closed subscheme
\[
  \mathcal Z=V(tx,x^2)\subset \Aff^1_x\times \Aff^1_t.
\]
The fibre over $t=a$ is defined by the ideal $(ax,x^2)\subset \CC[x]$. If $a\neq 0$, this ideal is $(x)$, so the fibre has length $1$. If $a=0$, the fibre is defined by $(x^2)$, so it has length $2$. Thus the length jumps in the family and the quotient $\CC[x,t]/(tx,x^2)$ has $t$-torsion and is not flat over $\CC[t]$.

For zero dimensional subschemes, the Hilbert polynomial is constant and equal to the length. Thus a family of length-$n$ subschemes should be a flat family whose fibres have Hilbert polynomial
\[
  P(m)=n.
\]
This is the case relevant to Hilbert schemes of points.

\subsection{The Hilbert functor and representability}

Let $X$ be a projective scheme over $\CC$. For a numerical polynomial $P$, meaning a polynomial with rational coefficients which takes integer values for all sufficiently large integers, the Hilbert functor
\[
  \mathcal{H}ilb^P_X:(\CC\text{-schemes})^{\mathrm{op}}\longrightarrow \mathrm{Sets}
\]
sends a test scheme $B$ to the set of closed subschemes
\[
  \mathcal Z\subset X\times B
\]
which are flat over $B$ and whose geometric fibres have Hilbert polynomial $P$. Here a geometric fibre means the fibre after base change to an algebraic closure of the residue field of the base point.

Grothendieck's representability theorem says that this functor is represented by a projective scheme $\Hilb^P(X)$ \cite{GrothendieckHilbert1961}. Equivalently, $\Hilb^P(X)$ carries a universal closed subscheme
\[
  \mathcal U\subset X\times \Hilb^P(X),
\]
flat over $\Hilb^P(X)$, such that every flat family
\[
  \mathcal Z\subset X\times B
\]
with Hilbert polynomial $P$ is obtained uniquely by pulling back $\mathcal U$ along a morphism
\[
  B\longrightarrow \Hilb^P(X).
\]
This is the fine moduli property of the Hilbert scheme. Nakajima recalls this statement in the form used for Hilbert schemes of points in \cite[Theorem~1.1]{Nakajima1999}.

For explicit constructions of Hilbert schemes of points, see Gustavsen--Laksov--Skjelnes \cite{GustavsenLaksovSkjelnes2007}, and for explicit affine open charts covering $\Hilb^n(\Aff^2)$, see Huibregtse \cite{Huibregtse2002}.

\subsection{Hilbert schemes of points}
\label{sec:hilbert-schemes-of-points}

When the Hilbert polynomial is the constant polynomial $n$, the corresponding Hilbert scheme is called the Hilbert scheme of $n$ points and is denoted
\[
  X^{[n]}=\Hilb^n(X).
\]
Its points are zero dimensional closed subschemes $Z\subset X$ of length
\[
  \length(Z)=\dim_\CC H^0(Z,\mathcal O_Z)=n.
\]
If $x_1,\ldots,x_n$ are distinct points of $X$, then the reduced union
\[
  Z=\{x_1,\ldots,x_n\}
\]
is one point of $X^{[n]}$. The point of using the Hilbert scheme is that it also includes the subschemes which appear when points collide. For example, a subscheme of length of two supported at one smooth point contains not only the point but also an infinitesimal tangent direction, and this is the basic difference between $X^{[n]}$ and the symmetric product $\operatorname{Sym}^n X$: the symmetric product records only the support points and their multiplicities, while the Hilbert scheme remembers the scheme structure of nonreduced limits. This comparison is made precise by the Hilbert--Chow morphism
\[
  \pi:X^{[n]}\longrightarrow \operatorname{Sym}^n X,
  \qquad
  Z\longmapsto \sum_{x\in X}\length(\mathcal O_{Z,x})[x].
\]
Over the open locus of $n$ distinct points, $\pi$ is an isomorphism. Over the diagonals, the symmetric product forgets the embedded and infinitesimal information that still remains valid for the Hilbert scheme \cite[Introduction and Theorem~1.5]{Nakajima1999}.

For the affine plane
\[
  S=\Aff^2_\CC=\Spec R,
  \qquad R=\CC[x,y],
\]
the notation becomes
\[
  S^{[n]}=\Hilb^n(S)=\Hilb^n(\Aff^2_\CC).
\]
Hence, in this affine case, a point of $S^{[n]}$ is the same thing as an ideal $I\subset R$ satisfying
\[
  \dim_\CC R/I=n.
\]
The universal family
\[
  \mathcal Z_n\subset S^{[n]}\times S
\]
has fibre over $[I]\in S^{[n]}$ equal to $\Spec(R/I)$.

\begin{remark}
Nakajima also gives a concrete construction of $S^{[n]}$ on the affine plane as a quotient of triples $(B_1,B_2,i)$, where $B_1,B_2\in \operatorname{End}_\CC(\CC^n)$ commute and $i:\CC\to \CC^n$ is a cyclic vector \cite[Theorem~1.9]{Nakajima1999}. From an ideal $I\subset \CC[x,y]$, one obtains $V=R/I$, the multiplication operators by $x$ and $y$, and the vector $1\in V$. Conversely, such data recover the ideal as the kernel of the map $f\mapsto f(B_1,B_2)i(1)$. This is also related to explicit constructions of Hilbert schemes of points, see Gustavsen--Laksov--Skjelnes \cite{GustavsenLaksovSkjelnes2007}. We do not use this construction in the paper, but it gives a concrete model for the moduli interpretation in the affine plane.
\end{remark}

\subsection{Deformations and embedded deformations}
\label{sec:deformations}

Deformation theory studies how a geometric object varies in a family, and its infinitesimal part represents the directions in which the object can move to first order. The basic method is to replace the base of a family by the spectrum of a small Artinian ring, so that a single morphism out of this spectrum captures a first-order variation. We follow Sernesi \cite[Chapter~1]{Sernesi2006}. In this paper, an \emph{Artinian local $\CC$-algebra} means a local Artinian $\CC$-algebra with residue field $\CC$. The smallest non-trivial such ring is the ring of dual numbers
\[
  A_\eps=\CC[\eps]/(\eps^2),
  \qquad D=\Spec A_\eps,
\]
which has a unique closed point and a one dimensional tangent direction recorded by $\eps$. A morphism out of $D$ is exactly a choice of first-order direction, and $D$ is the universal base for such directions.

\begin{definition}[Deformation of a scheme]
\label{def:deformation}
Let $X_0$ be a scheme over $\CC$ and let $A$ be an Artinian local $\CC$-algebra. A \emph{deformation of $X_0$ over $A$} is a scheme $\mathcal X$ together with a flat morphism $\mathcal X\to\Spec A$ and an isomorphism
\[
  \mathcal X\times_{\Spec A}\Spec\CC\;\xrightarrow{\ \sim\ }\;X_0
\]
of the closed fibre with $X_0$. When $A=A_\eps$, we call $\mathcal X$ a \emph{first-order} (or \emph{infinitesimal}) deformation of $X_0$.
\end{definition}

Flatness is what makes $\mathcal X$ a genuine family rather than an arbitrary scheme containing $X_0$: it forces the fibres to vary without jumps, so that the closed fibre $X_0$ is deformed rather than degenerated. Often one wants to deform not an abstract scheme but a subscheme of a fixed ambient space, keeping the ambient space rigid. This is the situation relevant to our Hilbert schemes.

\begin{definition}[Embedded deformation]
\label{def:embedded-deformation}
Let $Z\subset X$ be a closed subscheme of a scheme $X$ over $\CC$, and let $A$ be an Artinian local $\CC$-algebra. An \emph{embedded deformation of $Z$ in $X$ over $A$} is a closed subscheme
\[
  \mathcal Z\subset X\times\Spec A
\]
that is flat over $\Spec A$ and whose closed fibre satisfies
\[
  \mathcal Z\times_{\Spec A}\Spec\CC \;=\; Z\subset X.
\]
The ambient scheme $X$ is fixed. When $A=A_\eps$, we call $\mathcal Z$ a \emph{first-order embedded deformation} of $Z$.
\end{definition}

\begin{remark}
\label{rem:embedded-vs-abstract}
An embedded deformation remembers how $Z$ sits inside the fixed ambient $X$, not just the isomorphism class of $Z$. Forgetting the embedding turns an embedded deformation into a deformation of $Z$ in the sense of \Cref{def:deformation}, but two different embedded deformations may induce the same abstract deformation. The Hilbert scheme parametrises the embedded data, so it is embedded deformations that control its infinitesimal structure.
\end{remark}

This is precisely the infinitesimal content of the functor of points of the Hilbert scheme from \Cref{sec:hilbert-schemes-of-points}. A morphism $D\to\Hilb^{P}(X)$ whose closed point maps to $[Z]$ is the same thing as a flat family over $D$ of closed subschemes of $X$ restricting to $Z$, that is, a first-order embedded deformation of $Z$. In the affine situation of this paper, where $X=\mathbb A^2=\Spec R$ with $R=\CC[x,y]$ and $Z=\Spec(R/I)$, such a deformation is represented by an ideal
\[
  I_\eps\subset R\otimes_\CC A_\eps
\]
with $(R\otimes_\CC A_\eps)/I_\eps$ flat over $A_\eps$ and $I_\eps\otimes_{A_\eps}\CC=I$. We will make this correspondence explicit in the next two subsections.

\subsection{The dual numbers and the tangent space}
\label{sec:dual-numbers-tangent}

We now recall the Zariski tangent space and phrase it through the dual numbers, in the form that connects directly to the deformation theory above and to the torus action used in the later chapters.

\begin{definition}[Zariski tangent space]
\label{def:tangent-space}
Let $M$ be a scheme over $\CC$ and let $m\in M$ be a closed point with maximal ideal $\mathfrak m_m\subset\mathcal O_{M,m}$. The \emph{Zariski tangent space} of $M$ at $m$ is the $\CC$-vector space
\[
  T_mM=\bigl(\mathfrak m_m/\mathfrak m_m^{2}\bigr)^{\vee}.
\]
Equivalently, through the functor of points,
\[
  T_mM\;\cong\;\bigl\{\,\psi\in M(A_\eps)\;:\;\psi|_{\Spec\CC}=m\,\bigr\},
\]
the set of morphisms $D\to M$ whose restriction to the closed point of $D$ is $m$. The two definitions agree as $\CC$-vector spaces, see \cite[Chapter~II, Exercise~2.8]{Hartshorne1977} for the dual numbers description of the tangent space, and \cite[Section~2.2]{Sernesi2006} for the tangent space $t_F=F(A_\eps)$ of a functor of Artin rings.
\end{definition}

\begin{remark}[Functoriality and the torus action]
\label{rem:tangent-functorial}
The identification in \Cref{def:tangent-space} is functorial: a morphism of pointed schemes induces a $\CC$-linear map on tangent spaces, because it acts by composition on morphisms out of $D$. Consequently, if an algebraic group $G$ acts on $M$ and fixes the point $m$, then $G$ acts $\CC$-linearly on $T_mM$. This is the mechanism behind every weight computation in this paper. The torus $T=(\CC^\times)^2$ acts on $S^{[n]}$, and on the nested Hilbert scheme, fixing exactly the monomial ideals, hence $T$ acts linearly on the tangent space at each fixed point, and that tangent space splits into a direct sum of one dimensional weight spaces. Decomposing these representations is the subject of \Cref{chap:hilb-plane} and \Cref{chap:nested-tangent}. Under the identification $T_IS^{[n]}\cong\Hom_R(I,R/I)$ of \Cref{prop:tangent-hilb}, the induced action is
\[
  (\tau\cdot\varphi)(f)=\tau\cdot\varphi(\tau^{-1}\cdot f),
\qquad \tau\in T,
\]
which is the formula used to read off characters in \Cref{chap:hilb-plane}.
\end{remark}

Now we use \Cref{def:tangent-space} to study Hilbert schemes, combining it with \Cref{sec:deformations}. A morphism $D\to\Hilb^{P}(X)$ over $[Z]$ is a first-order embedded deformation of $Z$, so
\[
  T_{[Z]}\Hilb^{P}(X)\;\cong\;\bigl\{\text{first-order embedded deformations of }Z\bigr\}.
\]
For the affine plane this becomes more explicit: a tangent vector to $S^{[n]}$ at an ideal $I\subset R$ of colength $n$ is represented by an ideal
\[
  I_\eps\subset R\otimes_\CC A_\eps,
  \qquad
  I_\eps\otimes_{A_\eps}\CC=I,
\]
such that the quotient $(R\otimes_\CC A_\eps)/I_\eps$ is flat over $A_\eps$. Because $A_\eps$ is Artinian local, flatness of a finitely generated module is equivalent to freeness, so this quotient is a free $A_\eps$-module of rank $n$. The next subsection identifies the set of such ideals $I_\eps$ with the Hom space $\Hom_R(I,R/I)$.

\subsection{The tangent space to the Hilbert scheme}
\label{sec:tangent-hilb-general}

By \Cref{sec:dual-numbers-tangent}, the tangent space $T_IS^{[n]}$ is the set of first-order embedded deformations $I_\eps$ of $\Spec(R/I)$. The following proposition identifies this set with a Hom space, and it is the bridge between deformation theory and explicit calculations, see Sernesi's local Hilbert functor calculation \cite[Proposition~3.2.1]{Sernesi2006}.

\begin{proposition}[Tangent space to the Hilbert scheme]
\label{prop:tangent-hilb}
Let $I\subset R=\CC[x,y]$ be an ideal of colength $n$. Then
\[
  T_I S^{[n]}
  \cong
  \Hom_R(I,R/I).
\]
More generally, if $Z\subset X$ is a closed subscheme (\Cref{def:embedded-deformation}), then the first-order embedded deformations of $Z$ in $X$ are governed by
\[
  \Hom_{\mathcal O_X}(\mathcal I_Z,\mathcal O_Z)
  \;\cong\;
  H^0\bigl(Z,N_{Z/X}\bigr),
\]
the global sections of the normal sheaf $N_{Z/X}=\mathcal Hom_{\mathcal O_X}(\mathcal I_Z/\mathcal I_Z^{2},\mathcal O_Z)$, see \cite[Theorem~4.3.5]{Sernesi2006}.
\end{proposition}

\begin{proof}
We prove the affine statement, since this is the case this paper is looking at. Let $A=A_\eps=\CC[\eps]/(\eps^2)$ and $R_A=R\otimes_\CC A$. Suppose first that $I_A\subset R_A$ is a flat first-order deformation of $I$. For each $f\in I$, choose an element $g\in R$ such that
\[
  f+\eps g\in I_A.
\]
The class of $g$ modulo $I$ is independent of all choices. Indeed, if $f+\eps g$ and $f+\eps g'$ are both in $I_A$, then $\eps(g-g')\in I_A$. Let $Q=R_A/I_A$. Since $Q$ is a free $A$-module, the kernel of multiplication by $\eps$ on $Q$ is $\eps Q$. The class of $g-g'$ in $Q$ is represented by an element of $R\subset R_A$ and is killed by $\eps$, so its reduction modulo $\eps$ is zero. This means exactly that $g-g'\in I$. Define
\[
  \varphi_{I_A}(f)=\bar g\in R/I.
\]
A direct check shows that $\varphi_{I_A}$ is $R$-linear.

Conversely, given an $R$-linear map $\varphi:I\to R/I$, define
\[
  I_\varphi=
  \bigl\{\, f+\eps g\in R_A:
  f\in I \text{ and } \bar g=\varphi(f)\in R/I\,\bigr\}.
\]
The $R$-linearity of $\varphi$ implies that $I_\varphi$ is an ideal of $R_A$. Its special fibre is $I$. To see flatness, choose a $\CC$-vector space complement $W\subset R$ to $I$. Every class in $R_A/I_\varphi$ is represented uniquely by an element of $W\oplus \eps W$, so $R_A/I_\varphi$ is a free $A$-module with basis given by a basis of $W\cong R/I$. These two constructions are inverse to one another. Hence tangent vectors are precisely the elements of $\Hom_R(I,R/I)$.
\end{proof}

\begin{remark}
The proof above also fixes the sign convention used later. Replacing $\varphi$ by $-\varphi$ gives an equivalent identification of the tangent space with $\Hom_R(I,R/I)$, and the vector space and its torus weights are unchanged.
\end{remark}

\subsection{Smoothness of Hilbert schemes of points on surfaces}

The Hilbert scheme of points on a smooth surface is very well behaved.

\begin{theorem}[Fogarty \cite{Fogarty1968}]
\label{thm:fogarty}
Let $X$ be a smooth connected surface over an algebraically closed field. Then $X^{[n]}$ is smooth, connected, and irreducible of dimension $2n$.
\end{theorem}

For the affine plane this says that $S^{[n]}$ is smooth of dimension $2n$. In the later torus-fixed calculations, this dimension will appear again as the number of arm-leg tangent weights.

\subsection{Nested Hilbert schemes as incidence schemes}
\label{nhtdef}
The nested Hilbert scheme used in this paper is the incidence scheme whose points are pairs $Y\subset Z$ with lengths $n$ and $n+1$. This case where the lengths are different by one is the smooth nested Hilbert scheme studied by Cheah \cite[Theorem~3.2.2]{Cheah1998}. Since the later deformation formulas use ideal notation, we regard it as a subscheme
\[
  S^{[n,n+1]}
  \subset
  S^{[n+1]}\times S^{[n]}
\]
with coordinates $(Z,Y)$. In ideals, we write a point as
\[
  I\subset J\subset R,
  \qquad \dim_\CC R/I=n+1,
  \qquad \dim_\CC R/J=n.
\]
The letter $I$ always denotes the ideal of the larger scheme, and $J$ the ideal of the smaller scheme.

This incidence condition is closed in the product. It is therefore meaningful to study tangent vectors by starting with tangent vectors to the two Hilbert scheme factors and imposing a first-order compatibility condition. The resulting kernel formula is proved in \Cref{sec:nested-kernel}.

\section{Partitions, Young diagrams, and monomial ideals}
\label{chap:young-monomial}

This section introduces the combinatorial convention that is used in the paper.  The convention we use is the French convention: rows are drawn from bottom to top, and the box $(0,0)$ is the bottom-left box.

\subsection{Partitions and diagrams}

\begin{definition}
A partition of $n$ is a finite nonincreasing sequence of positive integers
\[
  \lambda=(\lambda_0,\lambda_1,\ldots,\lambda_{r-1}),
  \qquad
  \lambda_0\ge \lambda_1\ge \cdots\ge \lambda_{r-1}>0,
\]
such that
\[
  |\lambda|=\lambda_0+\cdots+\lambda_{r-1}=n.
\]
The Young diagram of $\lambda$ is the finite subset
\[
  D(\lambda)=\{(i,j)\in \NN^2:0\le j<r,\ 0\le i<\lambda_j\}.
\]
\end{definition}

Thus $j$ is the row index and $i$ is the column index.  For example, the partition $(3,2,1)$ is represented by
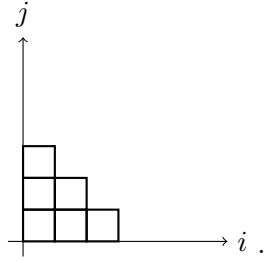
\[
  \drawyoungaxes{0/0,1/0,2/0,0/1,1/1,0/2}.
\]

\subsection{Staircases and monomial ideals}

Let $R=\CC[x,y]$.  A monomial ideal $I\subset R$ is determined by the set of exponent pairs
\[
  E(I)=\{(i,j)\in \NN^2:x^i y^j\in I\}.
\]
Since $I$ is an ideal, $E(I)$ is upward closed, which means that if $(i,j)\in E(I)$ and $(a,b)\in \NN^2$, then $(i+a,j+b)\in E(I)$.

The complement
\[
  B(I)=\NN^2\setminus E(I)
\]
is therefore downward closed.  If $I$ has finite colength, then $B(I)$ is finite.  We call this finite downward closed subset the staircase of $I$.

For background on monomial ideals, staircases, and their combinatorial resolutions, see Miller--Sturmfels \cite{MillerSturmfels2005}.

\begin{proposition}
\label{prop:partitions-monomial-ideals}
Monomial ideals of finite colength in $R=\CC[x,y]$ are in bijection with partitions.  The partition $\lambda$ corresponds to the ideal
\[
  I_\lambda=\bigl\langle x^i y^j:(i,j)\notin D(\lambda)\bigr\rangle.
\]
Moreover,
\[
  R/I_\lambda
  \cong
  \Span_\CC\{x^i y^j:(i,j)\in D(\lambda)\},
\]
and hence
\[
  \dim_\CC R/I_\lambda=|\lambda|.
\]
\end{proposition}

\begin{proof}
If $\lambda$ is a partition, then the complement of $D(\lambda)$ is upward closed, so the displayed monomials generate an ideal.  The monomials indexed by $D(\lambda)$ are exactly the monomials not contained in the ideal, and their residue classes form a basis of the quotient.

Conversely, let $I$ be a monomial ideal of finite colength.  The set $B(I)$ of monomials outside $I$ is finite and downward closed.  Define
\[
  \lambda_j=\#\{i:(i,j)\in B(I)\}.
\]
Downward closedness implies
\[
  \lambda_0\ge \lambda_1\ge \lambda_2\ge \cdots,
\]
and finite colength implies that only finitely many $\lambda_j$ are nonzero.  Removing the zero terms gives a partition $\lambda$, and $B(I)=D(\lambda)$.
\end{proof}

\subsection{Corners and addable boxes}

\begin{definition}
A box $c=(i,j)\in D(\lambda)$ is a removable corner if
\[
  (i+1,j)\notin D(\lambda)
  \qquad\text{and}\qquad
  (i,j+1)\notin D(\lambda).
\]
The set of removable corners is denoted $C(\lambda)$.
\end{definition}

Equivalently, $c$ is a removable corner if deleting it leaves the Young diagram of another partition.  For example, the partition $(3,2,1)$ has three removable corners:
\[
  (2,0),\qquad (1,1),\qquad (0,2).
\]
The middle corner is shown here:
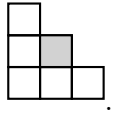
\[
  \drawyoungmarked{0/0,1/0,2/0,0/1,1/1,0/2}{1}{1}.
\]

\begin{definition}
A box $a\notin D(\lambda)$ is addable if $D(\lambda)\cup\{a\}$ is again the Young diagram of a partition.
\end{definition}

In the nested Hilbert scheme, it is usually more convenient to begin with the larger diagram $D(\lambda)$ and remove one corner.  If $c\in C(\lambda)$, we write
\[
  \lambda\setminus c
\]
for the partition whose diagram is $D(\lambda)\setminus\{c\}$.

\subsection{Arms and legs}

\begin{definition}
Let $s=(i,j)\in D(\lambda)$.  The arm length of $s$ is
\[
  a_\lambda(s)=\#\{(i',j)\in D(\lambda):i'>i\}.
\]
The leg length of $s$ is
\[
  \ell_\lambda(s)=\#\{(i,j')\in D(\lambda):j'>j\}.
\]
\end{definition}

Thus $a_\lambda(s)$ counts boxes strictly to the right of $s$, and $\ell_\lambda(s)$ counts boxes strictly above $s$.  For the partition $(3,2,1)$ and the box $s=(0,0)$, one has
\[
  a_\lambda(s)=2,
  \qquad
  \ell_\lambda(s)=2.
\]
For the box $s=(1,1)$, one has
\[
  a_\lambda(s)=0,
  \qquad
  \ell_\lambda(s)=0.
\]

\subsection{Minimal generators}

The minimal monomial generators of $I_\lambda$ lie along the outer boundary of $D(\lambda)$.  If
\[
  \lambda=(\lambda_0,\ldots,\lambda_{r-1}),
\]
then the ideal $I_\lambda$ is generated by the boundary monomials
\[
  x^{\lambda_0},\quad
  x^{\lambda_1}y,
  \quad \ldots,\quad
  x^{\lambda_{r-1}}y^{r-1},
  \quad y^r,
\]
after removing redundant generators.  For example,
\[
  I_{(3,2,1)}=(x^3,x^2y,xy^2,y^3),
\]
while
\[
  I_{(2,2)}=(x^2,y^2),
\]
because the boundary monomial $x^2y$ is a multiple of $x^2$.

This boundary description is useful for explicit computation of the tangent space, since a homomorphism $I_\lambda\to R/I_\lambda$ is determined by the images of these minimal generators, subject to the syzygies among them.

\section{Torus-fixed geometry of \texorpdfstring{$S^{[n]}$}{S[n]}}
\label{chap:hilb-plane}

We now study the torus-fixed geometry of the Hilbert scheme of points on the affine plane
\[
  S=\Aff^2_\CC=\Spec \CC[x,y].
\]
The torus action on $S$ makes $S^{[n]}$ into a $T$-variety.  At the fixed points, the tangent representation is described by the arm and leg lengths of the corresponding Young diagram.

\subsection{Torus action and weight convention}
\label{sec:torus-convention}

Let
\[
  T=(\CC^\times)^2.
\]
We use coordinates $(q,t)$ on $T$.  The torus acts on the affine plane by
\[
  (q,t)\cdot(a,b)=(qa,tb).
\]
This induces the contragredient action on the coordinate ring:
\[
  ((q,t)\cdot f)(x,y)=f(q^{-1}x,t^{-1}y).
\]
Thus, as a function, the monomial $x^iy^j$ has character $q^{-i}t^{-j}$.

\begin{convention}
All tangent weights in this paper are written for the above action on the affine plane.  Hence the tangent space to $\Aff^2$ at the origin has weights $q$ and $t$.  Equivalently, when a homomorphism sends a monomial with exponent $u$ to a monomial with exponent $v$, its tangent character is $q^{u_1-v_1}t^{u_2-v_2}$.  If one instead lets $T$ act directly on the coordinate ring by $x\mapsto qx$ and $y\mapsto ty$, all characters displayed below should be inverted.
\end{convention}

\subsection{Fixed points of the torus action}

\begin{proposition}
\label{prop:fixed-points-hilb}
The $T$-fixed points of $S^{[n]}$ are precisely the monomial ideals of finite colength in $R$.  Equivalently,
\[
  (S^{[n]})^T=\{I_\lambda:\lambda\vdash n\}.
\]
\end{proposition}

\begin{proof}
Every monomial ideal is fixed because the torus rescales monomials.  Conversely, suppose $I\subset R$ is fixed.  The coordinate ring decomposes into one dimensional weight spaces
\[
  R=\bigoplus_{(i,j)\in\NN^2}\CC x^iy^j.
\]
Since $T$ is diagonalizable, any $T$-stable subspace is a direct sum of weight spaces.  Therefore a $T$-stable ideal is spanned by monomials, hence is a monomial ideal.  The finite colength condition then gives a partition by \Cref{prop:partitions-monomial-ideals}.
\end{proof}

\subsection{The tangent representation}

The action of $T$ on $S$ induces an action on the Hilbert scheme $S^{[n]}$ by transport of subschemes. For the defining ideals, if $\tau\in T$ and $I\subset R$ is an ideal of colength $n$, then $\tau$ sends $I$ to the ideal
\[
  \tau\cdot I=\{\tau\cdot f:f\in I\}.
\]
Since $\tau$ acts on $R$ by algebra automorphisms, $\tau\cdot I$ is again an ideal, and since $\tau:S\to S$ is an isomorphism, the length of $R/I$ is preserved. Thus $S^{[n]}$ is a $T$-variety.

If $I\in S^{[n]}$ is a $T$-fixed point, then each $\tau\in T$ fixes $I$ and hence induces a linear automorphism of the Zariski tangent space $T_I S^{[n]}$ by differentiating the action at $I$. Therefore $T_I S^{[n]}$ is a representation of $T$. Under the identification
\[
  T_I S^{[n]}
  \cong
  \Hom_R(I,R/I),
\]
this representation is the induced action on the Hom space:
\[
  (\tau\cdot \varphi)(f)
  =
  \tau\cdot\varphi(\tau^{-1}\cdot f),
  \qquad
  \tau\in T,\quad f\in I.
\]
Since $I$ is $T$-stable, $\tau^{-1}\cdot f\in I$, and $R/I$ is a $T$-module. Since $T=(\CC^\times)^2$ is a torus, every finite dimensional $T$-representation decomposes into weight spaces. We can record the tangent representation by its $T$-character.

Let $\lambda\vdash n$. By \Cref{prop:tangent-hilb},
\[
  T_{I_\lambda}S^{[n]}
  \cong
  \Hom_R(I_\lambda,R/I_\lambda).
\]
The following theorem is the standard arm-leg formula for the character of this $T$-representation. The computation goes back to Ellingsrud--Str\o mme \cite{EllingsrudStromme1987}, Nakajima described the formula in arm-leg form \cite[Proposition~5.8]{Nakajima1999}, and Haiman uses the same arm-leg weights in the localization formula for the Hilbert scheme \cite[Section~3]{Haiman1998}. The proof below shows the translation to the torus convention used here.

\begin{theorem}[Arm-leg tangent weights, in our convention]
\label{thm:hilb-arm-leg}
Let $\lambda\vdash n$. The $T$-character of the tangent space at the fixed point $I_\lambda\in S^{[n]}$ is
\[
  \operatorname{ch}_T T_{I_\lambda}S^{[n]}
  =
  \sum_{s\in D(\lambda)}
  \left(
    q^{a_\lambda(s)+1}t^{-\ell_\lambda(s)}
    +
    q^{-a_\lambda(s)}t^{\ell_\lambda(s)+1}
  \right).
\]
In particular, the tangent space has $2n$ one dimensional weight summands.
\end{theorem}

\begin{proof}[Explanation of the convention]
The formula is the standard fixed-point tangent character. Nakajima writes the character in terms of variables $T_1,T_2$ and uses the convention
\[
  \sum_{s\in D(\lambda)}
  \left(
    T_1^{\ell_\lambda(s)+1}T_2^{-a_\lambda(s)}
    +
    T_1^{-\ell_\lambda(s)}T_2^{a_\lambda(s)+1}
  \right)
\]
\cite[Proposition~5.8]{Nakajima1999}. In this paper, the torus acts geometrically on $\Aff^2$ by
\[
  (q,t)\cdot(a,b)=(qa,tb),
\]
so the tangent weights of the horizontal and vertical coordinate directions are $q$ and $t$, respectively. Nakajima's formula is written with the first variable measuring the vertical direction and the second measuring the horizontal direction. Thus, in our notation,
\[
  T_1=t,
  \qquad
  T_2=q.
\]
Substitution gives
\[
  \sum_{s\in D(\lambda)}
  \left(
    t^{\ell_\lambda(s)+1}q^{-a_\lambda(s)}
    +
    t^{-\ell_\lambda(s)}q^{a_\lambda(s)+1}
  \right).
\]
The final statement follows because the sum contains two one dimensional characters for each of the $n$ boxes of $D(\lambda)$, matching the smoothness of $S^{[n]}$ of dimension $2n$ introduced in \Cref{thm:fogarty}.
\end{proof}

\begin{remark}
The same computation is often expressed using significant cleft pairs or using arrows in the Young diagram, see Chaput--Evain \cite[Section~2.1]{ChaputEvain2015}. The important point for the nested Hilbert scheme is that this arrow basis is compatible with the torus action and can be modified locally near a chosen corner.
\end{remark}

\subsection{The arrows as deformation}
\label{sec:ordinary-arrow-deformations}

The arrow in the Young diagram is a way to describe homomorphisms
\[
  I_\lambda\longrightarrow R/I_\lambda .
\]
A homogeneous map is determined by where it sends the monomial generators on the boundary of $I_\lambda$. In the diagram, an arrow represents that a boundary monomial is sent to the residue class of a monomial inside $D(\lambda)$. The condition that the arrow defines an $R$-linear map is exactly the condition that it respects the syzygies among the boundary generators. Thus the valid arrows represent elements of
\[
  \Hom_R(I_\lambda,R/I_\lambda).
\]
By \Cref{prop:tangent-hilb}, these homomorphisms are precisely the first-order embedded deformations of the subscheme defined by $I_\lambda$. The tail of an arrow lies on the monomial boundary of the ideal $I_\lambda$, and the head lies on a monomial that survives in the quotient $R/I_\lambda$.  The slope of the arrow records the torus character of the corresponding homomorphism.

For example, take $\lambda=(4,2,1)$ and the box $s=(1,0)$.  Then $a(s)=2$ and $\ell(s)=1$.  The two ordinary arrows attached to $s$ may be pictured as
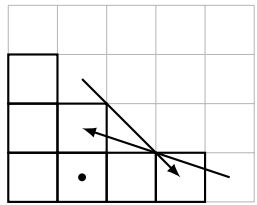
\[
\begin{tikzpicture}[scale=0.65,baseline={(0,-0.2)}]
  \draw[gray!55] (0,0) grid (5,4);
  \foreach \x/\y in {0/0,1/0,2/0,3/0,0/1,1/1,0/2}{\draw[thick] (\x,\y) rectangle ++(1,1);}
  \fill[black] (1.5,0.5) circle (0.08);
  \draw[->,>=latex,thick] (1.5,2.5) -- (3.5,0.5);
  \draw[->,>=latex,thick] (4.5,0.5) -- (1.5,1.5);
\end{tikzpicture}
\]
The two characters are
\[
  q^{a(s)+1}t^{-\ell(s)}=q^3t^{-1},
  \qquad
  q^{-a(s)}t^{\ell(s)+1}=q^{-2}t^2.
\]
The ordinary Hilbert scheme is already controlled by a deformation theory rule on Young diagrams: boxes describe the quotient $R/I_\lambda$, boundary monomials describe generators of $I_\lambda$, and arrows describe first-order maps from the ideal to the quotient.  The nested Hilbert scheme will impose one more condition on these arrows, which is the compatibility with the inclusion of ideals.

\subsection{Example: the partition \texorpdfstring{$(2,1)$}{(2,1)}}

Let $\lambda=(2,1)$.  Then
\[
  D(\lambda)=\{(0,0),(1,0),(0,1)\},
  \qquad
  I_\lambda=(x^2,xy,y^2).
\]
The arms and legs are
\[
\begin{array}{c|c|c}
 s & a_\lambda(s) & \ell_\lambda(s) \\
\midrule
 (0,0) & 1 & 1 \\
 (1,0) & 0 & 0 \\
 (0,1) & 0 & 0
\end{array}
\]
Therefore
\[
\operatorname{ch}_T T_{I_\lambda}S^{[3]}
=
q^2t^{-1}+q^{-1}t^2+2q+2t.
\]
The six weights agree with the dimension $2\cdot 3=6$.

\subsection{The tautological bundle}

The universal family
\[
  \mathcal Z_n\subset S^{[n]}\times S
\]
is finite flat of degree $n$ over $S^{[n]}$. Let
\[
  p:S^{[n]}\times S\longrightarrow S^{[n]}
\]
be the projection. The tautological bundle is
\[
  \mathcal V_n=p_*\mathcal O_{\mathcal Z_n}.
\]
It is a vector bundle of rank $n$, and its fibre at $[I]\in S^{[n]}$ is
\[
  (\mathcal V_n)_{[I]}\cong R/I.
\]
For the case of affine plane see Nakajima
\cite[Section~4.3]{Nakajima1999}, and for tautological sheaves on Hilbert schemes
of points on surfaces see Lehn \cite{Lehn1999}.

The universal family $\mathcal Z_n$ introduced here is also the central object of the next chapter.  In \Cref{chap:nested-geometry} we realise the nested Hilbert scheme $S^{[n,n+1]}$ both as a modification of the universal family $\mathcal Z_{n+1}$ over $S^{[n+1]}$ and as the blow-up of $S^{[n]}\times S$ along $\mathcal Z_n$, and we read off its torus-fixed points from the same Young diagram combinatorics.

\section{The nested Hilbert scheme and the universal family}
\label{chap:nested-geometry}

In this section we study the global geometry behind the nested Hilbert scheme
\[
  S^{[n,n+1]}=\Hilb^{n,n+1}(\Aff^2).
\]
There are two natural maps out of $S^{[n,n+1]}$.  One map goes to the universal family over $S^{[n+1]}$ and is birational.  The other goes to $S^{[n]}\times S$ and realizes the nested Hilbert scheme as a blow-up along the universal family over $S^{[n]}$.  The two maps are closely related, but they have different local geometric structure.

\subsection{Definition and residual point}

A point of $S^{[n,n+1]}$ is a pair of ideals
\[
  I\subset J\subset R,
  \qquad
  \dim_\CC R/I=n+1,
  \qquad
  \dim_\CC R/J=n.
\]
There is an exact sequence
\[
  0\longrightarrow J/I\longrightarrow R/I\longrightarrow R/J\longrightarrow 0.
\]
Since the lengths of $R/I$ and $R/J$ differ by one, $J/I$ is an $R$-module of length one. Hence it is supported at a unique closed point $p\in S$, and
\[
  J/I\cong R/\mathfrak m_p
\]
as an $R$-module, where $\mathfrak m_p$ is the maximal ideal of $p$.  This point is called the residual point of the nested pair and is denoted
\[
  \res(I,J)=p.
\]
Geometrically, $I$ defines the larger subscheme $Z$ of length $n+1$, $J$ defines the smaller subscheme $Y$ of length $n$, and the quotient $J/I$ is the residual piece of length one of $Z$ not contained in $Y$.

\begin{theorem}[Cheah]
\label{thm:cheah-smooth}
If $S$ is a smooth surface, then $S^{[n,n+1]}$ is smooth and irreducible of dimension $2n+2$.
\end{theorem}

This is the smooth nested case in Cheah's classification of smooth nested Hilbert schemes \cite[Theorem~3.2.2]{Cheah1998}.  In the torus-fixed computations later, this smoothness will be a dimension check: since $S^{[n,n+1]}$ is smooth of dimension $2n+2$, the tangent space at each fixed point should have dimension $2n+2$. The tangent character formula will give exactly $2n+2$ one dimensional $T$-weight spaces.

\subsection{The map to the universal family over \texorpdfstring{$S^{[n+1]}$}{S[n+1]}}
\label{sec:birational-universal-family}

Let
\[
  \mathcal Z_{n+1}\subset S^{[n+1]}\times S
\]
be the universal family.  A point of $\mathcal Z_{n+1}$ is a pair $(Z,p)$, where $Z\in S^{[n+1]}$ and $p\in \Supp(Z)$.

There is a natural morphism
\[
  \rho:S^{[n,n+1]}\longrightarrow \mathcal Z_{n+1}
\]
given by
\[
  \rho(I\subset J)=(I,\res(I,J)).
\]
Here the first component is the larger subscheme of length $(n+1)$, and the second component is the residual point.

\begin{proposition}[Birational map to the universal family]
\label{prop:rho-birational}
The morphism
\[
  \rho:S^{[n,n+1]}\longrightarrow \mathcal Z_{n+1}
\]
is birational.
\end{proposition}

\begin{proof}
Let
\[
  \operatorname{HC}_{n+1}:S^{[n+1]}\longrightarrow \operatorname{Sym}^{n+1}S
\]
be the Hilbert--Chow morphism recalled in \Cref{sec:hilbert-schemes-of-points}. Let
$W\subset \operatorname{Sym}^{n+1}S$ be the open locus of reduced cycles
\[
  [p_0]+\cdots+[p_n],
  \qquad p_i\neq p_j \text{ for } i\neq j.
\]
This is the complement of the image of the big diagonal in $S^{n+1}$, hence is dense and open. Put
\[
  U=\operatorname{HC}_{n+1}^{-1}(W)\subset S^{[n+1]}.
\]
Over $W$, the Hilbert--Chow morphism is an isomorphism, since a reduced
subscheme of length $(n+1)$ is uniquely determined by its support cycle
\cite[Theorem~1.5]{Nakajima1999}. Thus $U$ is the dense open locus of
subschemes
\[
  Z=\{p_0,\ldots,p_n\}
\]
with $p_0,\ldots,p_n$ distinct.

Let
\[
  \mathcal V=\mathcal Z_{n+1}\times_{S^{[n+1]}}U
  \subset \mathcal Z_{n+1}
\]
be the universal family restricted to $U$. A point of $\mathcal V$ is a pair
$(Z,p_i)$, where $Z=\{p_0,\ldots,p_n\}$ is reduced and $p_i\in Z$. Over $U$,
the universal family
\[
  \mathcal Z_{n+1}|_U\longrightarrow U
\]
is finite \'etale of degree $n+1$: it is a finite flat family whose fibres are reduced
sets of $n+1$ distinct points. After pulling this family back to $\mathcal V$, the marked point gives a
tautological section. Since a section of a finite \'etale morphism has an image that is open and closed, its complement is again finite flat, now of degree $n$. Fibrewise this
complement is
\[
  Z\setminus\{p_i\}.
\]
Therefore deleting the marked point defines a morphism
\[
  \sigma:\mathcal V\longrightarrow S^{[n,n+1]},
  \qquad
  (Z,p_i)\longmapsto (Z\setminus\{p_i\}\subset Z).
\]

By construction, $\rho\circ\sigma=\mathrm{id}_{\mathcal V}$. Conversely, if a
nested pair maps to $(Z,p_i)\in\mathcal V$, then $Z$ is a reduced set of
$n+1$ distinct points, and the only subscheme of length $n$ of $Z$ with residual
point $p_i$ is $Z\setminus\{p_i\}$. Hence $\sigma\circ\rho$ is the identity over
$\mathcal V$. Thus $\rho$ restricts to an isomorphism
\[
  \rho^{-1}(\mathcal V)\cong \mathcal V.
\]
Since $\mathcal Z_{n+1}\to S^{[n+1]}$ is finite flat and $U$ is dense open,
$\mathcal V$ is dense open in $\mathcal Z_{n+1}$, and since
$S^{[n,n+1]}$ is irreducible by \Cref{thm:cheah-smooth}, this proves that $\rho$ is birational.
\end{proof}

\begin{remark}
The map to the universal family is birational, not an isomorphism in general.  The universal family remembers the residual point, while the nested Hilbert scheme remembers the residual point together with the way a piece of length one sits inside the local algebra of the larger subscheme. In the fibre description in \Cref{prop:rho-fiber-socle}, we will explain this difference.
\end{remark}

\begin{example}
Consider the length-three subscheme of $\Aff^2$ defined by
\[
  I=(x,y)^2=(x^2,xy,y^2).
\]
Then
\[
  R/I=\Span_\CC\{1,x,y\}
\]
is supported at the origin. We now describe the intermediate ideals
\[
  I\subset J,
  \qquad
  \dim_\CC R/J=2.
\]
Such an ideal $J$ is the same as a one dimensional $R$-submodule
\[
  J/I\subset R/I.
\]

Let $L\subset R/I$ be a one dimensional vector subspace. Write a nonzero
element of $L$ as
\[
  v=a+bx+cy
  \qquad
  \text{mod } I.
\]
For $L$ to be an $R$-submodule, it must be stable under multiplication by
$x$ and $y$. Since
\[
  x\cdot v=ax,
  \qquad
  y\cdot v=ay
  \qquad
  \text{mod } I,
\]
if $a\neq 0$, then $L$ would have to contain both $x$ and $y$, which is
impossible for a line. Hence $a=0$, and so
\[
  L\subset \Span_\CC\{x,y\}.
\]
Conversely, every line in $\Span_\CC\{x,y\}$ is killed by both $x$ and $y$
modulo $I$, and is therefore an $R$-submodule of $R/I$.

Thus intermediate ideals $I\subset J$ with $\dim_\CC R/J=2$ are parametrized
by lines
\[
  L\subset \Span_\CC\{x,y\}.
\]
If
\[
  L=\Span_\CC\{ax+by\}
\]
for some nonzero pair $(a,b)$, then the corresponding ideal is
\[
  J=I+(ax+by).
\]
Therefore these intermediate ideals are parametrized by
\[
  \PP(\Span_\CC\{x,y\})\cong \PP^1.
\]

For every such $J$, the quotient $J/I$ is killed by the maximal ideal
$(x,y)$, so its residual support is the origin. Therefore all these nested pairs
map under $\rho$ to the same point
\[
  (I,0)\in \mathcal Z_3.
\]
Hence, the map
$\rho$ is not injective on closed points. In particular, $\rho$ is not an
isomorphism.
\end{example}

\subsection{Socle fibres of the map to the universal family}
\label{sec:socle-fibres}

The previous counterexample to isomorphism is a motivation for this subsection. We now study the fibres of the map
\[
  \rho:S^{[n,n+1]}\longrightarrow \mathcal Z_{n+1}
\]
from a local algebra angle. Let $I\subset R$ be an ideal of finite length and let
$p\in \operatorname{Supp}(R/I)$. Denote by $\mathfrak m_p\subset R$ the maximal
ideal corresponding to $p$.

\begin{definition}
The local socle of $R/I$ at $p$ is
\[
  \Soc_p(R/I)
  =
  \{\bar f\in R/I:\mathfrak m_p\bar f=0\}.
\]
\end{definition}

Thus $\Soc_p(R/I)$ is the subspace of elements of $R/I$ killed by all functions
vanishing at $p$. It measures the possible submodules of length one of the local
algebra $R/I$ supported at $p$. The following proposition is a modification of the discussion of the morphism that appeared in \cite[Section~3.5]{Haiman2001}.

\begin{proposition}[Socle fibre of the map to the universal family]
\label{prop:rho-fiber-socle}
Let $(I,p)\in \mathcal Z_{n+1}$, so that $I\subset R$ has colength $n+1$ and
$p\in \operatorname{Supp}(R/I)$. Then the fibre of
\[
  \rho:S^{[n,n+1]}\longrightarrow \mathcal Z_{n+1}
\]
over $(I,p)$ is naturally identified with
\[
  \rho^{-1}(I,p)
  \cong
  \PP(\Soc_p(R/I)),
\]
where $\PP(V)$ denotes the projective space of lines in the vector space $V$.
\end{proposition}

\begin{proof}
A point in the fibre over $(I,p)$ is an ideal $J$ such that
\[
  I\subset J,
  \qquad
  \dim_\CC R/J=n,
  \qquad
  \operatorname{Supp}(J/I)=\{p\}.
\]
Since $\dim_\CC R/I=n+1$, the condition $\dim_\CC R/J=n$ is equivalent to
\[
  \dim_\CC J/I=1.
\]
Thus such a $J$ is the same thing as a one dimensional vector subspace
\[
  L=J/I\subset R/I
\]
which is an $R$-submodule and is supported at $p$.

The condition that $L$ is supported at $p$ is exactly the condition that
$\mathfrak m_p$ kills $L$:
\[
  \mathfrak m_pL=0.
\]
Indeed, a one dimensional $R$-module supported at $p$ is isomorphic to
$R/\mathfrak m_p$ as an $R$-module. Therefore $L$ is a line contained in
\[
  \Soc_p(R/I)
  =
  \{\bar f\in R/I:\mathfrak m_p\bar f=0\}.
\]

Conversely, let $L\subset \Soc_p(R/I)$ be a line. Since $\mathfrak m_pL=0$, for
every $r\in R$ the element $r-r(p)$ lies in $\mathfrak m_p$ and therefore acts
trivially on $L$. Hence $r$ acts on $L$ by the scalar $r(p)$, so $L$ is an
$R$-submodule of $R/I$. Let $J\subset R$ be the inverse image of $L$ under the
quotient map
\[
  R\longrightarrow R/I.
\]
Then $J$ is an ideal, $I\subset J$, and $J/I=L$ has dimension one and is supported
at $p$. Hence $J$ defines a point of the fibre $\rho^{-1}(I,p)$.

This gives a bijection between intermediate ideals in the fibre and lines in
$\Soc_p(R/I)$. This bijection upgrades to an isomorphism of schemes. The fibre represents the functor sending a test scheme $T$ to the line subbundles $\mathcal L\subset (R/I)\otimes_\CC\mathcal O_T$ with $\mathfrak m_p\mathcal L=0$, that is, line subbundles of the trivial bundle annihilated by every element of $\mathfrak m_p$. Multiplication by a fixed element of $\mathfrak m_p$ is a constant linear endomorphism of $R/I$, so the condition $\mathfrak m_pL=0$ is the vanishing of finitely many linear maps on the tautological line $\mathcal O(-1)\subset (R/I)\otimes\mathcal O$ over $\PP(R/I)$. Each such condition is a section of $(R/I)\otimes\mathcal O(1)$, hence linear in the homogeneous coordinates. These linear equations cut out the reduced linear subspace
\[
  \PP(\Soc_p(R/I))\subset \PP(R/I),
\]
and the fibre is naturally isomorphic to $\PP(\Soc_p(R/I))$ as a projective
scheme.
\end{proof}

In this paper, for monomial ideals, the socle has a direct Young diagram interpretation.

\begin{proposition}[Corners are monomial socle directions]
\label{prop:corners-socle}
Let $\lambda$ be a partition and let $I_\lambda\subset R$ be the corresponding
monomial ideal. Then $I_\lambda$ is supported at the origin, and
\[
  \Soc_0(R/I_\lambda)
  =
  \Span_\CC\{x^iy^j:(i,j)\in C(\lambda)\},
\]
where $C(\lambda)$ is the set of removable corners of $D(\lambda)$. In particular,
the torus-invariant (not necessarily point-wise torus-fixed) socle lines are exactly the lines spanned by the removable
corner monomials.
\end{proposition}

\begin{proof}
The quotient $R/I_\lambda$ has monomial basis
\[
  \{x^iy^j:(i,j)\in D(\lambda)\}.
\]
A basis monomial $x^iy^j$ is killed by $x$ modulo $I_\lambda$ precisely when
\[
  (i+1,j)\notin D(\lambda),
\]
and it is killed by $y$ modulo $I_\lambda$ precisely when
\[
  (i,j+1)\notin D(\lambda).
\]
Thus $x^iy^j$ is killed by the maximal ideal $(x,y)$ precisely when $(i,j)$ is a
removable corner of $D(\lambda)$.

Since $R/I_\lambda$ has a monomial basis, a linear combination is killed by both
$x$ and $y$ exactly when all its monomial terms are killed by both $x$ and $y$.
Therefore
\[
  \Soc_0(R/I_\lambda)
  =
  \Span_\CC\{x^iy^j:(i,j)\in C(\lambda)\}.
\]
Finally, distinct monomials have distinct $T$-characters. Hence the
torus-invariant lines in the socle are precisely the coordinate lines spanned by
the removable corner monomials.
\end{proof}

\subsection{The blow-up description and its fibres}
\label{sec:blowup-description}

In this subsection, we let $\PP(V)$ denote the projective space of lines in the vector space $V$. Equivalently, $\PP(V^\vee)$ parametrizes one dimensional quotients of $V$.

There is another natural map in contrast to \Cref{prop:rho-birational}, but this time we keep the smaller subscheme and the residual point:
\[
  \phi:S^{[n,n+1]}\longrightarrow S^{[n]}\times S,
  \qquad
  \phi(I\subset J)=(J,\res(I,J)).
\]
The universal family
\[
  \mathcal Z_n\subset S^{[n]}\times S
\]
consists of those pairs $(J,p)$ with $p\in \Supp(R/J)$.

\begin{theorem}[Blow-up description]
\label{thm:blowup}
For a smooth connected surface $S$, the morphism
\[
  \phi:S^{[n,n+1]}\longrightarrow S^{[n]}\times S
\]
is the blow-up of $S^{[n]}\times S$ along the universal family $\mathcal Z_n$:
\[
  S^{[n,n+1]}
  \cong
  \Bl_{\mathcal Z_n}(S^{[n]}\times S).
\]
\end{theorem}

\begin{remark}
We do not reproduce the proof here. The statement is the standard
Ellingsrud--Str\o mme blow-up description of the nested Hilbert
scheme with length difference one. In the form used in this paper, it was introduced by Lehn
\cite[Section~1.2]{Lehn1999}. Ryan--Yang state the same result as
\cite[Proposition~3.5]{Ryan2020}, where the morphism
\[
  X^{[n+1,n]}\longrightarrow X^{[n]}\times X
\]
is identified with the blow-up of $X^{[n]}\times X$ along the universal family
$Z^{[n]}$. 
\end{remark}

\begin{proposition}[Fibre of the blow-up map]
\label{prop:phi-fiber-generators}
Let $(J,p)\in S^{[n]}\times S$. In the affine case, the fibre of
\[
  \phi:S^{[n,n+1]}\to S^{[n]}\times S
\]
over $(J,p)$ is naturally
\[
  \PP(J(p)^\vee),
  \qquad
  J(p)=J\otimes_R R/\mathfrak m_p\cong J/\mathfrak m_pJ,
\]
where $\PP(V)$ denotes lines in $V$. Equivalently, this is the projective space of
one dimensional quotients of $J(p)$.
\end{proposition}

\begin{proof}
A point of the fibre over $(J,p)$ is an ideal $I\subset J$ such that the residual
quotient is supported at $p$:
\[
  J/I\cong R/\mathfrak m_p.
\]
Equivalently, it is a surjective $R$-linear map
\[
  J\twoheadrightarrow R/\mathfrak m_p
\]
with kernel $I$, considered up to multiplication by a nonzero scalar. Since
$\mathfrak m_p$ acts trivially on $R/\mathfrak m_p$, every such map factors
uniquely through
\[
  J/\mathfrak m_pJ.
\]
Thus the fibre is the projective space of one dimensional quotients of
$J/\mathfrak m_pJ$.

Conversely, any one dimensional quotient of $J/\mathfrak m_pJ$ gives, after
composition with the quotient map $J\to J/\mathfrak m_pJ$, a quotient
\[
  J\twoheadrightarrow R/\mathfrak m_p
\]
up to scalar. Let $I$ be its kernel. Then $I\subset J$ and
\[
  J/I\cong R/\mathfrak m_p.
\]
The exact sequence
\[
  0\to J/I\to R/I\to R/J\to 0
\]
shows that
\[
  \dim_\CC R/I=\dim_\CC R/J+\dim_\CC J/I=n+1.
\]
Hence $I\subset J$ defines a point of $S^{[n,n+1]}$ lying over $(J,p)$.
This construction is inverse to the previous one, and gives the stated projective
space description.
\end{proof}

If $p\notin \Supp(R/J)$, then $J_p=R_p$, hence
\[
  J(p)\cong R/\mathfrak m_p,
\]
so the fibre is a point. If $p\in \Supp(R/J)$ and the ideal $J$ requires $i$ local
generators at $p$, then
\[
  \dim_\CC J/\mathfrak m_pJ=i
\]
by Nakayama's lemma, and the fibre is a projective space of dimension $i-1$.
This is the local geometry the blow-up shows.

\begin{proposition}[Addable boxes as local generator directions]
\label{prop:addable-generator-directions}
Let $J=I_\mu$ be a monomial ideal of finite colength, so that the corresponding
subscheme is supported at the origin. Then
\[
  I_\mu/(x,y)I_\mu
\]
has a basis given by the images of the minimal monomial generators of $I_\mu$.
Equivalently, these basis vectors are indexed by the addable boxes of $D(\mu)$.
Consequently, the torus-fixed points in the fibre $\phi^{-1}(I_\mu,0)$ correspond
to addable boxes of $\mu$.
\end{proposition}

\begin{proof}
We use Nakayama's lemma in the following form. If $(A,\mathfrak m,k)$
is a local ring and $M$ is a finitely generated $A$-module, then a set of elements
of $M$ generates $M$ if and only if their images span the $k$-vector space
$M/\mathfrak mM$. In particular, the minimal number of generators of $M$ is
\[
  \dim_k M/\mathfrak mM.
\]

Apply this to the local ring
\[
  A=R_{(x,y)}
\]
at the origin, with maximal ideal $\mathfrak m=(x,y)A$, and to the local
$A$-module $J_A=J\otimes_R A$. Since
\[
  J_A/\mathfrak m J_A
  \cong
  J/(x,y)J,
\]
the vector space $J/(x,y)J$ records the minimal generators of $J$ at the origin.

Now assume $J=I_\mu$ is monomial. The quotient $J/(x,y)J$ has a basis given by
the residue classes of the minimal monomial generators of $J$. Indeed, a monomial
of $J$ becomes zero in $J/(x,y)J$ precisely when it is divisible by $x$ or $y$
times another monomial already lying in $J$, and these are exactly the non-minimal
monomial generators. Thus the surviving monomial classes are the
minimal monomial generators of $I_\mu$.

A monomial outside $D(\mu)$ is a minimal generator of $I_\mu$ when
adding the corresponding box to $D(\mu)$ gives another Young diagram. Hence
the minimal monomial generators of $I_\mu$ are indexed by the addable
boxes of $D(\mu)$.

By \Cref{prop:phi-fiber-generators}, the fibre of $\phi$ over $(I_\mu,0)$ is the
projective space of one dimensional quotients of
\[
  I_\mu/(x,y)I_\mu.
\]
The torus acts diagonally on the monomial generator classes, and distinct
monomials have distinct $T$-characters. Therefore the torus-fixed one dimensional
quotients are precisely the coordinate quotients dual to the monomial basis
vectors. These coordinate quotients are indexed by the addable boxes of $D(\mu)$.

Choosing the quotient corresponding to an addable box $a$ gives the kernel ideal
\[
  I_{\mu\cup\{a\}}\subset I_\mu.
\]
Therefore the associated nested fixed point is
\[
  I_{\mu\cup\{a\}}\subset I_\mu,
\]
as claimed.
\end{proof}

\subsection{Adding a free point and adding an infinitesimal point}
\label{sec:add-point-vs-infinitesimal}

The two fibre descriptions show how the universal family and the blow-up have the same deformation geometry. A point of $S^{[n]}\times S$ is a pair $(Y,p)$ consisting of a subscheme of length $n$ and a point of the surface.  If
\[
  p\notin \Supp(Y),
\]
then there is only one way to form a subscheme of length $(n+1)$ whose residual point is $p$, which is to take the disjoint union
\[
  Z=Y\sqcup \{p\}.
\]
Algebraically, this is the transverse intersection of the ideal of $Y$ with the maximal ideal $\mathfrak m_p$.

However, when $p\in \Supp(Y)$, adding one more unit of length is no longer the same as adding an independent reduced point.  By \Cref{prop:phi-fiber-generators}, one must choose a one dimensional quotient of
\[
  J(p)=J/\mathfrak m_pJ,
\]
where $J$ is the ideal of $Y$. This quotient specifies which local generator direction of $J$ is removed when
passing to the smaller ideal $I\subset J$. Since smaller ideals define larger
subschemes, this choice determines how one extra unit of length is added at $p$.
Thus the exceptional fibre of the blow-up is not an accidental extra component, instead
it records the infinitesimal choices that appear when the new point collides with
the old subscheme.

Dually, if one starts with the larger subscheme $Z$ defined by $I$ and asks which subscheme of length $n$ has been removed, \Cref{prop:rho-fiber-socle} says that the choice is a line in the local socle of $R/I$.  In monomial coordinates the two descriptions become the two sides of the same Young diagram operation:
\[
\begin{array}{c|c|c}
\toprule
\text{map} & \text{local fibre} & \text{monomial fixed data} \\
\midrule
S^{[n,n+1]}\to S^{[n]}\times S
& \text{quotients of } I_\mu/(x,y)I_\mu
& \text{addable boxes of }\mu \\
S^{[n,n+1]}\to \mathcal Z_{n+1}
& \text{lines in }\Soc_0(R/I_\lambda)
& \text{removable corners of }\lambda \\
\bottomrule
\end{array}
\]
This table is one of the main organizing points of the paper: for a nested Hilbert scheme of difference one, viewed from the larger subscheme, it is a socle direction, represented by a removable corner, while viewed from the smaller subscheme, it is a quotient of the local generator space, represented by an addable box.

\subsection{Fixed points of the nested Hilbert scheme}

\begin{corollary}
\label{cor:fixed-points-nested}
The $T$-fixed points of $S^{[n,n+1]}$ are indexed by pairs
\[
  (\lambda,c),
  \qquad
  \lambda\vdash n+1,
  \qquad
  c\in C(\lambda).
\]
The corresponding nested ideal pair is
\[
  I_\lambda\subset I_{\lambda\setminus c}.
\]
\end{corollary}

\begin{proof}
A $T$-fixed point of the product $S^{[n+1]}\times S^{[n]}$ consists of two monomial ideals by \Cref{prop:fixed-points-hilb} applied to each factor.  Thus a fixed point of the nested incidence scheme has the form
\[
  I_\lambda\subset I_\mu,
\]
with $|\lambda|=n+1$ and $|\mu|=n$.  The quotient $I_\mu/I_\lambda$ is one dimensional, so the diagrams differ by exactly one box:
\[
  D(\mu)=D(\lambda)\setminus\{c\}.
\]
For $D(\mu)$ to be a Young diagram, the removed box must be a removable corner of $D(\lambda)$.  Conversely, every removable corner gives such an inclusion.  This agrees with the socle line of \Cref{prop:corners-socle}, and at the same time it agrees with the addable box of \Cref{prop:addable-generator-directions} applied to $\mu=\lambda\setminus c$.
\end{proof}

\section{Tangent spaces and torus weights for the nested Hilbert scheme}
\label{chap:nested-tangent}

The nested Hilbert scheme is an incidence scheme inside a product of two Hilbert schemes as in \Cref{nhtdef}.  Its tangent space is therefore obtained by imposing a compatibility condition on first-order deformations of the two ideals.  This kernel description is well studied for nested Hilbert schemes, for example see Cheah's tangent space discussion \cite[Section~0.4]{Cheah1998}, Can's kernel formula \cite[Section~5]{Can2007}, and the formulation of Chaput--Evain \cite[Lemma~12]{ChaputEvain2015}.  At torus-fixed points, this compatibility condition is the deformation theory reason behind the shortening rule for Young diagram arrows in \Cref{sec:shortening-rule}.

\subsection{The compatibility kernel}
\label{sec:nested-kernel}

Let
\[
  I\subset J\subset R,
  \qquad
  \dim_\CC R/I=n+1,
  \qquad
  \dim_\CC R/J=n.
\]
Let
\[
  \iota:I\hookrightarrow J
\]
be the inclusion and
\[
  \pi:R/I\longrightarrow R/J
\]
be the natural quotient map.

\begin{theorem}[Tangent space to the nested Hilbert scheme]
\label{thm:nested-kernel}
There is an isomorphism
\[
T_{(I,J)}S^{[n,n+1]}
\cong
\ker\left(
\Hom_R(I,R/I)\oplus \Hom_R(J,R/J)
\xrightarrow{\delta}
\Hom_R(I,R/J)
\right),
\]
where
\[
  \delta(\alpha,\beta)=\pi\circ \alpha-\beta\circ\iota.
\]
\end{theorem}

\begin{proof}
A tangent vector to the product $S^{[n+1]}\times S^{[n]}$ at $(I,J)$ is a pair
\[
  (\alpha,\beta)
  \in
  \Hom_R(I,R/I)\oplus \Hom_R(J,R/J),
\]
by \Cref{prop:tangent-hilb}.  We must determine when the corresponding first-order deformations preserve the inclusion $I\subset J$.

Let $A_\eps=\CC[\eps]/(\eps^2)$ and $R_\eps=R\otimes_\CC A_\eps$.  The map $\alpha$ corresponds to a deformation $I_\alpha\subset R_\eps$, and $\beta$ corresponds to a deformation $J_\beta\subset R_\eps$.  The pair defines a tangent vector to the incidence scheme precisely when
\[
  I_\alpha\subset J_\beta.
\]
For $f\in I$, choose representatives so that
\[
  f+\eps g\in I_\alpha,
  \qquad
  \bar g=\alpha(f)\in R/I.
\]
The same element belongs to $J_\beta$ exactly when its image in $R/J$ agrees with the image prescribed by the deformation of $J$.  This condition is
\[
  \pi(\alpha(f))=\beta(f)
  \qquad\text{in } R/J,
\]
where $f$ is regarded as an element of $J$ via the inclusion $\iota:I\hookrightarrow J$.  Thus the nested condition is exactly
\[
  \pi\circ\alpha=\beta\circ\iota.
\]
This is equivalent to $\delta(\alpha,\beta)=0$.
\end{proof}

\begin{remark}
The map $\delta$ is $T$-equivariant whenever $(I,J)$ is $T$-fixed.  Hence the kernel decomposes into $T$-weight spaces.  This is what makes the Young diagram weight calculation possible and valuable.
\end{remark}

\subsection{Fixed-point notation}

By \Cref{cor:fixed-points-nested}, the $T$-fixed points of $S^{[n,n+1]}$ are indexed by a partition $\lambda\vdash n+1$ together with a removable corner $c\in C(\lambda)$.  Fix such a pair and put
\[
  \mu=\lambda\setminus c.
\]
The corresponding fixed point of $S^{[n,n+1]}$ is
\[
  (I_\lambda,I_\mu),
  \qquad
  I_\lambda\subset I_\mu.
\]
Write
\[
  c=(i_0,j_0).
\]
For a box $s=(i,j)\in D(\lambda)$, we say:
\begin{itemize}
  \item $s$ is \emph{left of $c$} if $j=j_0$ and $i<i_0$.
  \item $s$ is \emph{below $c$} if $i=i_0$ and $j<j_0$.
\end{itemize}
We will see that these are the only boxes whose ordinary tangent weights are modified by the nested condition.

\subsection{A monomial criterion for arrows}
\label{sec:monomial-criterion-arrows}

The arrow construction can be made algebraic using only generators and syzygies of monomial ideals.  We introduce the following criterion because it is the mechanism behind the compatibility kernel and the Macaulay2 verification in \Cref{sec:computational-verification}.

This is a lcm-syzygy description of homomorphisms from a monomial ideal. In two variables, the minimal generators form a boundary chain, and the adjacent boundary syzygies generate the first syzygy module, see Miller--Sturmfels \cite[Chapter~3]{MillerSturmfels2005}.

\begin{lemma}[Monomial syzygy criterion]
\label{lem:monomial-syzygy-criterion}
Let $I=(m_1,\ldots,m_r)$ be a monomial ideal, where $m_1,\ldots,m_r$ are the minimal monomial generators ordered along the boundary, and let $K\subset R$ be another monomial ideal.  A collection of classes
\[
  u_i\in R/K,
  \qquad i=1,\ldots,r,
\]
defines an $R$-linear homomorphism $\varphi:I\to R/K$ with $\varphi(m_i)=u_i$ if and only if for every pair of these minimal generators one has
\[
  \frac{\operatorname{lcm}(m_i,m_j)}{m_i}u_i
  =
  \frac{\operatorname{lcm}(m_i,m_j)}{m_j}u_j
  \qquad\text{in }R/K.
\]
It is enough to impose these equations for adjacent minimal generators along the monomial boundary.
\end{lemma}

\begin{proof}
The displayed relations are exactly the images of the standard monomial syzygies
\[
  \frac{\operatorname{lcm}(m_i,m_j)}{m_i}m_i
  -
  \frac{\operatorname{lcm}(m_i,m_j)}{m_j}m_j=0.
\]
Thus they are necessary.  Conversely, the first syzygies of a monomial ideal are generated by these lcm syzygies.  For a monomial ideal of finite colength in two variables, the minimal generators form one boundary chain, and the adjacent lcm syzygies along this chain generate all the remaining lcm syzygies.  Therefore a choice of images satisfying the adjacent relations descends from the free module on the minimal generators to the ideal $I$ itself, giving the desired homomorphism.
\end{proof}

\subsection{The row-and-column shortening rule}
\label{sec:shortening-rule}

Before stating the formula, we isolate the local extension problem created by the corner.  The lemma is included to make the compatibility kernel explanation algebraic rather than only combinatorial.

\begin{lemma}[Extension across the corner]
\label{lem:extension-across-corner}
Let $\lambda\vdash n+1$, let $c=(i_0,j_0)\in C(\lambda)$, put $\mu=\lambda\setminus c$, and write
\[
  m_c=x^{i_0}y^{j_0}.
\]
Then
\[
  I_\mu=I_\lambda+(m_c),
  \qquad
  R/I_\mu=(R/I_\lambda)/\CC m_c.
\]
Let $\rho_c:R/I_\lambda\to R/I_\mu$ be the quotient map.  For a homomorphism
\[
  \alpha:I_\lambda\longrightarrow R/I_\lambda,
\]
there exists a homomorphism
\[
  \beta:I_\mu\longrightarrow R/I_\mu
\]
satisfying
\[
  \beta|_{I_\lambda}=\rho_c\circ\alpha
\]
if and only if there is an element $u\in R/I_\mu$ such that
\[
  xu=\rho_c\bigl(\alpha(xm_c)\bigr),
  \qquad
  yu=\rho_c\bigl(\alpha(ym_c)\bigr).
\]
When such a $u$ is chosen, the extension is given by $\beta(m_c)=u$ together with the prescribed restriction to $I_\lambda$.
\end{lemma}

\begin{proof}
Since $c$ is a removable corner, both $xm_c$ and $ym_c$ lie in $I_\lambda$.  The quotient
\[
  I_\mu/I_\lambda
\]
is one dimensional, generated by the class of $m_c$, and is killed by the maximal ideal $(x,y)$.  Equivalently, the map $R\to I_\mu/I_\lambda$ sending $1$ to the class of $m_c$ identifies $I_\mu/I_\lambda$ with $R/(x,y)$, so the kernel is generated by $x$ and $y$.  Thus $I_\mu$ is generated by $I_\lambda$ and $m_c$, and all relations involving the new generator are generated, modulo relations already in $I_\lambda$, by
\[
  xm_c\in I_\lambda,
  \qquad
  ym_c\in I_\lambda.
\]
If $\beta$ is an extension of $\rho_c\circ\alpha$ and $u=\beta(m_c)$, then $R$-linearity gives
\[
  xu=\beta(xm_c)=\rho_c(\alpha(xm_c)),
  \qquad
  yu=\beta(ym_c)=\rho_c(\alpha(ym_c)).
\]
Conversely, suppose that such a $u$ exists.  Define $\beta$ on the generators $I_\lambda$ and $m_c$ by
\[
  \beta(f)=\rho_c(\alpha(f))\quad(f\in I_\lambda),
  \qquad
  \beta(m_c)=u.
\]
The two displayed equations say precisely that the relations $xm_c\in I_\lambda$ and $ym_c\in I_\lambda$ are respected.  The remaining relations are relations inside $I_\lambda$, and these are respected because $\alpha$ is $R$-linear.  Hence the assignment descends to an $R$-linear homomorphism $I_\mu\to R/I_\mu$.
\end{proof}

We will also use the standard arrow basis for the ordinary Hilbert scheme in the following boundary form.  For every box $s=(i,j)\in D(\lambda)$ there are two homogeneous arrows, denoted
\[
  H_s^\lambda,
  \qquad
  V_s^\lambda,
\]
forming a basis of $\Hom_R(I_\lambda,R/I_\lambda)$.  Their weights are
\[
  \operatorname{wt}(H_s^\lambda)=q^{a(s)+1}t^{-\ell(s)},
  \qquad
  \operatorname{wt}(V_s^\lambda)=q^{-a(s)}t^{\ell(s)+1}.
\]
Algebraically, these arrows are the homogeneous solutions of the adjacent lcm-syzygy equations of \Cref{lem:monomial-syzygy-criterion}.  The horizontal arrow $H_s^\lambda$ is the solution whose critical boundary relation is the horizontal one at the end of the row through $s$, while the vertical arrow $V_s^\lambda$ is the solution whose critical boundary relation is the vertical one at the top of the column through $s$.  This is the usual arm-leg arrow basis for the tangent space of $S^{[|\lambda|]}$ \cite[Proposition~5.8]{Nakajima1999}, and see also Chaput--Evain \cite[Section~2.1]{ChaputEvain2015}.

The next theorem is the standard fixed-point tangent weight formula for the smooth nested Hilbert scheme of difference one.  We state it in the torus convention of \Cref{sec:torus-convention}.  The formula is the arm-leg character for $S^{[n+1]}$, modified near the distinguished corner.  It is equivalent to Cheah's explicit tangent basis computation for smooth nested Hilbert schemes \cite[Section~2.6]{Cheah1998}, to the staircase and cleft pair computation of the tangent representation by Chaput--Evain \cite[Section~2.2]{ChaputEvain2015}, and to the row-and-column shortening rule recorded by Koncki--Zielenkiewicz \cite[Section~5.3]{KonckiZielenkiewicz2025}.

\begin{theorem}[Cheah, Chaput--Evain, Koncki--Zielenkiewicz, in our convention]
\label{thm:nested-weights}
Let $(I_\lambda,I_{\lambda\setminus c})$ be a $T$-fixed point of $S^{[n,n+1]}$, where $\lambda\vdash n+1$ and $c=(i_0,j_0)$ is a removable corner.  For each box $s\in D(\lambda)$, set $a(s)=a_\lambda(s)$ and $\ell(s)=\ell_\lambda(s)$.  Define two weights
\[
  w_1(s)=
  \begin{cases}
    q^{a(s)}t^{-\ell(s)}, & \text{if }s\text{ is left of }c,\\
    q^{a(s)+1}t^{-\ell(s)}, & \text{otherwise,}
  \end{cases}
\]
and
\[
  w_2(s)=
  \begin{cases}
    q^{-a(s)}t^{\ell(s)}, & \text{if }s\text{ is below }c,\\
    q^{-a(s)}t^{\ell(s)+1}, & \text{otherwise.}
  \end{cases}
\]
Then
\[
  \operatorname{ch}_T T_{(I_\lambda,I_{\lambda\setminus c})}S^{[n,n+1]}
  =
  \sum_{s\in D(\lambda)}\bigl(w_1(s)+w_2(s)\bigr).
\]
In particular,
\[
  \dim_\CC T_{(I_\lambda,I_{\lambda\setminus c})}S^{[n,n+1]}=2(n+1)=2n+2.
\]
\end{theorem}

\begin{proof}
The tangent-basis statement is the cited theorem for the smooth nested Hilbert scheme.  Cheah constructs explicit weight bases for the smooth nested schemes $Z_{n-1,n}(S)$ \cite[Section~2.6]{Cheah1998}, Chaput--Evain express the same tangent representation using staircases and cleft pairs \cite[Section~2.2]{ChaputEvain2015}, and Koncki--Zielenkiewicz record the resulting character as the row-and-column shortening rule for the added box \cite[Section~5.3]{KonckiZielenkiewicz2025}.  Translating their notation to our convention, the added box of the smaller diagram $\mu=\lambda\setminus c$ is the corner $c$ of $\lambda$, and our geometric torus action gives the displayed powers of $q$ and $t$.

We now explain how the same rule is read from the compatibility kernel of \Cref{thm:nested-kernel}.  Put $\mu=\lambda\setminus c$ and $m_c=x^{i_0}y^{j_0}$.  As ideals,
\[
  I_\mu=I_\lambda+(m_c),
\]
although the minimal monomial boundary changes locally: the new boundary generator $m_c$ is inserted and adjacent redundant boundary generators may disappear.  Let
\[
  \rho_c:R/I_\lambda\longrightarrow R/I_\mu
\]
be the quotient map.  By \Cref{lem:extension-across-corner}, a deformation $\alpha:I_\lambda\to R/I_\lambda$ can be paired with a deformation $\beta:I_\mu\to R/I_\mu$ precisely when the projected map $\rho_c\alpha$ extends across the new generator $m_c$.  This extension is controlled by the two equations
\begin{equation}
\label{eq:corner-extension-equations}
  xu=\rho_c\alpha(xm_c),
  \qquad
  yu=\rho_c\alpha(ym_c),
\end{equation}
where $u=\beta(m_c)$.  These are exactly the two local adjacent syzygy conditions created by inserting $m_c$ into the boundary.

We spell out this local shift in terms of sources and targets of the ordinary arrows.  For $s=(i,j)\in D(\lambda)$, the horizontal arrow $H_s^\lambda$ has source
\[
  x^{i+a(s)+1}y^j
\]
on the right boundary of the row through $s$, and target monomial
\[
  x^iy^{j+\ell(s)}.
\]
Thus its weight is $q^{a(s)+1}t^{-\ell(s)}$.  Similarly, the vertical arrow $V_s^\lambda$ has source
\[
  x^iy^{j+\ell(s)+1}
\]
and target monomial
\[
  x^{i+a(s)}y^j,
\]
so its weight is $q^{-a(s)}t^{\ell(s)+1}$.

Now suppose that $s=(i,j_0)$ lies to the left of the distinguished corner $c=(i_0,j_0)$.  Then $a(s)=i_0-i$, and the source of the ordinary horizontal arrow is
\[
  x^{i+a(s)+1}y^{j_0}
  =
  x^{i_0+1}y^{j_0}
  =
  xm_c.
\]
After passing from $I_\lambda$ to $I_\mu=I_\lambda+(m_c)$, the new generator $m_c$ is inserted one step before $xm_c$ on this part of the boundary.  The corresponding homogeneous nested arrow is therefore read with source $m_c$ and with the same target monomial
\[
  x^iy^{j_0+\ell(s)}.
\]
The values on the rest of the boundary are then propagated by the same adjacent syzygies as in the ordinary arrow construction.  More precisely, by \Cref{lem:monomial-syzygy-criterion} the homomorphism $\beta$ is determined by its values on the minimal generators of $I_\mu$ subject to the adjacent lcm-syzygies along the boundary chain. Away from the inserted generator $m_c$ this chain and its syzygies coincide with those of $I_\lambda$, so the values of $\beta$ on the remaining boundary generators are forced from $\beta(m_c)$ by exactly the relations that propagate the ordinary arrow $H_s^\lambda$, and the only relation that changes is the local pair at $m_c$ recorded by \eqref{eq:corner-extension-equations}.  Indeed, if $\beta(m_c)=x^iy^{j_0+\ell(s)}$ in $R/I_\mu$, then the first equation in \eqref{eq:corner-extension-equations} says that the value at the old boundary monomial $xm_c$ is obtained by multiplying this value by $x$.  The second equation imposes the adjacent compatibility in the other direction. For this horizontal arrow the left-hand side is zero because $y\cdot x^iy^{j_0+\ell(s)}$ lies outside the diagram, and the corresponding value of $\rho_c\alpha(ym_c)$ is zero as well.  Thus the weight of the modified horizontal arrow is
\[
  \frac{x^{i_0}y^{j_0}}{x^iy^{j_0+\ell(s)}}
  =
  q^{i_0-i}t^{-\ell(s)}
  =
  q^{a(s)}t^{-\ell(s)}.
\]
This is the ordinary horizontal weight divided by $q$.

The corner itself is a small exception to the phrase ``left of the corner''.  For $s=c$, the ordinary horizontal arrow has source $xm_c$ and target $m_c$.  Since $\rho_c(m_c)=0$ in $R/I_\mu$, this arrow is compatible with $\beta(m_c)=0$ and its weight remains $q$.  For all other horizontal arrows not lying to the left of $c$, the local relation controlling the arrow is not shifted by the insertion of $m_c$, so the horizontal weight remains
\[
  q^{a(s)+1}t^{-\ell(s)}.
\]

The vertical arrows are identical with $x$ and $y$ interchanged.  If $s=(i_0,j)$ lies below $c$, then $\ell(s)=j_0-j$, and the source of the ordinary vertical arrow is
\[
  x^{i_0}y^{j+\ell(s)+1}
  =
  x^{i_0}y^{j_0+1}
  =
  ym_c.
\]
The insertion of $m_c$ moves this vertical boundary relation one step earlier.  The modified vertical arrow has source $m_c$ and target monomial
\[
  x^{i_0+a(s)}y^j,
\]
with the remaining boundary values again propagated by the adjacent syzygies.  The second equation in \eqref{eq:corner-extension-equations} records the compatibility with the old value at $ym_c$.  Its weight is therefore
\[
  \frac{x^{i_0}y^{j_0}}{x^{i_0+a(s)}y^j}
  =
  q^{-a(s)}t^{j_0-j}
  =
  q^{-a(s)}t^{\ell(s)}.
\]
This is the ordinary vertical weight divided by $t$.

For $s=c$, the ordinary vertical arrow has source $ym_c$ and target $m_c$, which is killed by $\rho_c$, hence its weight remains $t$.  For all other vertical arrows not lying below $c$, the local relation is unchanged and the weight remains
\[
  q^{-a(s)}t^{\ell(s)+1}.
\]
When a weight occurs with multiplicity, this arrow description is understood
with respect to the standard arrow basis fixed in the cited tangent basis
construction \cite{Cheah1998}. The character formula itself is independent of this choice of basis. Thus the compatibility kernel explains exactly the row-and-column modifications stated in the theorem.  The tangent basis construction proves that the corresponding modified arrows form a basis of the tangent space. Equivalently, there are two weights for each of the $|\lambda|=n+1$ boxes.  The dimension also agrees with Cheah's smoothness theorem as in \Cref{thm:cheah-smooth}.  This gives the character formula in the convention in our paper.
\end{proof}

\begin{remark}
The basis construction used at the start of the proof is the standard tangent basis theorem of Cheah, Chaput--Evain, and Koncki--Zielenkiewicz.  The role of the compatibility kernel argument is to place their row-and-column formula in the deformation theory framework of this paper: the only new local condition imposed by the incidence relation is the extension across the corner generator $m_c$, which is encoded by \eqref{eq:corner-extension-equations}.
\end{remark}

For instance, let $\lambda=(4,2,1)$ and let the distinguished corner be $c=(1,1)$, shown in gray below.  The box $s=(1,0)$ lies below $c$.  The ordinary vertical arrow attached to $s$ must be shortened by one step in the nested tangent space:
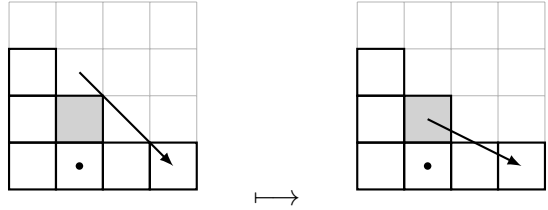
\[
\begin{array}{ccc}
\begin{tikzpicture}[scale=0.62,baseline={(0,-0.2)}]
  \draw[gray!55] (0,0) grid (4,4);
  \fill[gray!35] (1,1) rectangle (2,2);
  \foreach \x/\y in {0/0,1/0,2/0,3/0,0/1,1/1,0/2}{\draw[thick] (\x,\y) rectangle ++(1,1);}
  \fill[black] (1.5,0.5) circle (0.08);
  \draw[->,>=latex,thick] (1.5,2.5) -- (3.5,0.5);
\end{tikzpicture}
&\quad\longmapsto\quad&
\begin{tikzpicture}[scale=0.62,baseline={(0,-0.2)}]
  \draw[gray!55] (0,0) grid (4,4);
  \fill[gray!35] (1,1) rectangle (2,2);
  \foreach \x/\y in {0/0,1/0,2/0,3/0,0/1,1/1,0/2}{\draw[thick] (\x,\y) rectangle ++(1,1);}
  \fill[black] (1.5,0.5) circle (0.08);
  \draw[->,>=latex,thick] (1.5,1.5) -- (3.5,0.5);
\end{tikzpicture}
\end{array}
\]
The gray box represents the corner added to $\mu$ (equivalently, removed from $\lambda$), and the black dot marks the box $s$.  The right-hand diagram expresses the same compatibility condition as the kernel formula: the ordinary arrow uses the old boundary relation through $ym_c$, while the shortened arrow uses the new generator $m_c$.

\subsection{Comparison with the ordinary Hilbert scheme}

The nested weight formula can be summarized as follows.  Begin with the character of $T_{I_\lambda}S^{[n+1]}$ from the arm-leg formula of \Cref{thm:hilb-arm-leg} (with $n$ replaced by $n+1$):
\[
  \sum_{s\in D(\lambda)}
  \left(
    q^{a(s)+1}t^{-\ell(s)}+q^{-a(s)}t^{\ell(s)+1}
  \right).
\]
Then:
\begin{itemize}
  \item for boxes left of the corner $c$, replace $q^{a(s)+1}t^{-\ell(s)}$ by $q^{a(s)}t^{-\ell(s)}$.
  \item for boxes below the corner $c$, replace $q^{-a(s)}t^{\ell(s)+1}$ by $q^{-a(s)}t^{\ell(s)}$.
  \item leave all other weights unchanged.
\end{itemize}
Thus the nested tangent space has the same number of weights as the tangent space to $S^{[n+1]}$, but some of the weights are shifted by one unit in the distinguished row or column.

This is consistent with the global geometry: both $S^{[n+1]}$ and $S^{[n,n+1]}$ have dimension $2n+2$, but the nested scheme remembers extra incidence data, so the local torus representation is not generally the same.

\section{Examples and computational verification}
\label{chap:examples-computation}

This section works out several fixed-point tangent characters and explains the finite Macaulay2 \cite{Macaulay2} verification included with the paper source.  We check the shortening rule that the kernel formula of \Cref{thm:nested-kernel} produces the same weights as \Cref{thm:nested-weights}.

\subsection{The partition \texorpdfstring{$(2,1)$}{(2,1)}}

Let
\[
  \lambda=(2,1),
  \qquad
  D(\lambda)=\{(0,0),(1,0),(0,1)\}.
\]
The removable corners are
\[
  c_1=(0,1),
  \qquad
  c_2=(1,0).
\]
The ordinary Hilbert-scheme weights at $I_\lambda\in S^{[3]}$ are
\[
  q^2t^{-1},\ q^{-1}t^2,
  \ q,\ t,
  \ q,\ t.
\]

\subsubsection*{Removing the upper corner}

Take $c_1=(0,1)$.  The box below $c_1$ is $(0,0)$, and there are no boxes left of $c_1$.  Hence only the second weight attached to $(0,0)$ changes:
\[
  q^{-1}t^2\rightsquigarrow q^{-1}t.
\]
Thus
\[
  \operatorname{ch}_T T_{(I_{(2,1)},I_{(2)})}S^{[2,3]}
  =
  q^2t^{-1}+q^{-1}t+2q+2t.
\]

\subsubsection*{Removing the right corner}

Take $c_2=(1,0)$.  The box left of $c_2$ is $(0,0)$, and there are no boxes below $c_2$.  Hence only the first weight attached to $(0,0)$ changes:
\[
  q^2t^{-1}\rightsquigarrow qt^{-1}.
\]
Thus
\[
  \operatorname{ch}_T T_{(I_{(2,1)},I_{(1,1)})}S^{[2,3]}
  =
  qt^{-1}+q^{-1}t^2+2q+2t.
\]
In both cases the dimension is six, as expected for $S^{[2,3]}$.

\subsection{A less symmetric example}
\label{sec:less-symmetric-example}

Let
\[
  \lambda=(3,2,1),
  \qquad
  c=(1,1).
\]
Then $|\lambda|=6$, so the corresponding fixed point lies in $S^{[5,6]}$.  The corner $c$ is the middle removable corner:
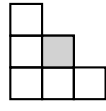
\[
  \drawyoungmarked{0/0,1/0,2/0,0/1,1/1,0/2}{1}{1}.
\]
The boxes left of $c$ and below $c$ are respectively
\[
  (0,1),
  \qquad
  (1,0).
\]
The ordinary weights at $I_\lambda\in S^{[6]}$ are listed in the middle columns of the table, and the nested weights after shortening are listed on the right.
\[
\begin{array}{c|c|c|c|c}
\toprule
\text{box }s & a(s) & \ell(s) & \text{ordinary weights} & \text{nested weights} \\
\midrule
(0,0) & 2 & 2 & q^3t^{-2},\ q^{-2}t^3 & q^3t^{-2},\ q^{-2}t^3 \\
(1,0) & 1 & 1 & q^2t^{-1},\ q^{-1}t^2 & q^2t^{-1},\ q^{-1}t \\
(2,0) & 0 & 0 & q,\ t & q,\ t \\
(0,1) & 1 & 1 & q^2t^{-1},\ q^{-1}t^2 & qt^{-1},\ q^{-1}t^2 \\
(1,1) & 0 & 0 & q,\ t & q,\ t \\
(0,2) & 0 & 0 & q,\ t & q,\ t \\
\bottomrule
\end{array}
\]
The character is therefore
\[
\begin{aligned}
\operatorname{ch}_T T_{(I_\lambda,I_{\lambda\setminus c})}S^{[5,6]}
={}&q^3t^{-2}+q^{-2}t^3+q^2t^{-1}+q^{-1}t  \\
&+qt^{-1}+q^{-1}t^2+3q+3t.
\end{aligned}
\]
There are twelve weights, equal to $2|\lambda|=12$.

\subsection{Computational verification}
\label{sec:computational-verification}

We also include a Macaulay2 code which checks the nested weight formula by computing the kernel in \Cref{thm:nested-kernel} directly from monomial generators and syzygies. We summarize the mathematical design here. The complete source is available at \url{https://github.com/chenyangzhaoicl/master_thesis_nested_hilbert_scheme_deformation}.

The input of the computation is a pair $(\lambda,c)$, where $\lambda$ is a partition and $c\in C(\lambda)$ is a removable corner. The script forms
\[
  \mu=\lambda\setminus c
\]
and the two monomial ideals
\[
  I_\lambda\subset I_\mu\subset R.
\]
Monomials are represented by exponent pairs $(i,j)$, so the monomial $x^iy^j$ is stored as the pair $(i,j)$. This makes divisibility, least common multiples, quotient bases, and torus weights completely explicit.

For a monomial ideal $I=(g_1,\ldots,g_r)$ and another monomial ideal $K$, the code constructs the vector space $\Hom_R(I,R/K)$ by assigning unknown coefficients to the possible images of the minimal generators $g_i$ in the monomial basis of $R/K$. It then imposes the lcm-syzygy equations
\[
  \frac{\operatorname{lcm}(g_i,g_j)}{g_i}\varphi(g_i)
  -
  \frac{\operatorname{lcm}(g_i,g_j)}{g_j}\varphi(g_j)
  =0
  \qquad\text{in }R/K.
\]
This is the computational version of \Cref{lem:monomial-syzygy-criterion}. The implementation imposes the pairwise lcm-syzygies, so in particular it includes the adjacent boundary syzygies used in the Young diagram arrow description.

The nested tangent space is then obtained by constructing the two ordinary Hom spaces
\[
  \Hom_R(I_\lambda,R/I_\lambda),
  \qquad
  \Hom_R(I_\mu,R/I_\mu),
\]
and imposing the compatibility equations
\[
  \pi\circ\alpha=\beta\circ\iota
\]
from \Cref{thm:nested-kernel}. Thus the linear system computed by the code is exactly the system defining
\[
  \ker\left(
  \Hom_R(I_\lambda,R/I_\lambda)
  \oplus
  \Hom_R(I_\mu,R/I_\mu)
  \longrightarrow
  \Hom_R(I_\lambda,R/I_\mu)
  \right).
\]

The torus weights are computed from the same convention used in the paper. If a variable sends a source monomial $x^iy^j$ to a target monomial $x^ry^s$, then its additive weight is
\[
  (i-r,j-s),
\]
corresponding to the character $q^{i-r}t^{j-s}$. The syzygy and compatibility equations are homogeneous for this grading. Therefore the full linear system splits into independent blocks indexed by torus weights. For each weight, the multiplicity is computed as
\[
  \#\{\text{variables of that weight}\}
  -
  \operatorname{rank}\{\text{equations of that weight}\},
\]
using exact rational linear algebra.

The script compares this computed multiset of weights with a second multiset obtained independently from the shortening formula in \Cref{thm:nested-weights}. The comparison therefore checks two different descriptions of the same tangent representation: the compatibility kernel that comes from deformation theory and the Young diagram shortening rule.

\subsubsection*{Worked trace for the less symmetric example}

For the example
\[
  \lambda=(3,2,1),
  \qquad
  c=(1,1),
  \qquad
  \mu=(3,1,1),
\]
the code prints the quotient bases and minimal generators used to build the linear system, and then prints two multisets of additive torus weights. The relevant output is
\[
\begin{array}{l}
\texttt{kernel side      = \{-1,1 => 1, -1,2 => 1, -2,3 => 1, 0,1 => 3,}\\
\texttt{                    1,-1 => 1, 1,0 => 3, 2,-1 => 1, 3,-2 => 1\}}\\
\texttt{modification side = \{-1,1 => 1, -1,2 => 1, -2,3 => 1, 0,1 => 3,}\\
\texttt{                    1,-1 => 1, 1,0 => 3, 2,-1 => 1, 3,-2 => 1\}}\\
\texttt{agree = true.}
\end{array}
\]
The printed notation means
\[
  \texttt{a,b => m}
  \qquad\Longleftrightarrow\qquad
  \text{$m$ copies of the character }q^at^b.
\]
Thus the displayed multiset gives
\[
  q^3t^{-2}+q^{-2}t^3+q^2t^{-1}+q^{-1}t
  +qt^{-1}+q^{-1}t^2+3q+3t,
\]
which agrees with the character computed by hand in \Cref{sec:less-symmetric-example}.

\begin{example}[Finite Macaulay2 check]
The Macaulay2 computation verifies the equality between the compatibility kernel weights and the shortening weights for every partition $\lambda$ with $|\lambda|\leq 16$ and every removable corner $c\in C(\lambda)$.

For the bound $|\lambda|\leq 16$, running
\[
  \texttt{M2 -{}-script code/verify\_weights.m2}
\]
generates the output
\[
\begin{array}{l}
\texttt{All tests passed for |lambda| <= 16.}\\
\texttt{Checked 2455 nested fixed points.}
\end{array}
\]
The first line says that the computed weight multiset agrees with the shortening rule multiset for every tested pair $(\lambda,c)$ with $|\lambda|\leq 16$. The second line records the size of the finite test: there are $2455$ pairs consisting of a partition of size at most $16$ and a removable corner. This automated check includes all partitions with a maximal size including those where exchanging the two coordinates or shortening the wrong arrow would change the answer. For the convenience of the execution time we put an adjustable maximal value for $|\lambda|$.
\end{example}

\bibliographystyle{alpha}
\bibliography{references}

\end{document}